\documentclass[12pt,reqno]{amsart}
\usepackage[nospace,noadjust]{cite}
\usepackage{verbatim,amscd,amssymb}
\usepackage{graphicx}
\usepackage{amsmath}
\usepackage{amsmath}
\usepackage{amsfonts}
\usepackage{amssymb}
\usepackage{inputenc}
\usepackage{enumitem}
\usepackage{float}

\usepackage[active]{srcltx}

\graphicspath{ {./images/} }

\usepackage{fontenc}

\usepackage{float}
\newtheorem{theorem}{Theorem}[section]
\newtheorem{lemma}[theorem]{Lemma}

\theoremstyle{definition}
\newtheorem{definition}[theorem]{Definition}
\newtheorem{example}[theorem]{Example}

\theoremstyle{remark}

\theoremstyle{assumption}

\numberwithin{equation}{section}

\newcommand{\RA}{\rightarrow}

\newcommand{\sha}{\succ\mkern-14mu_s\;}

\begin{document}
	
	\date{\today}
	
	\title[Forcing among exact patterns of triods]{Forcing among exact patterns of triods}
	
	%\dedicatory{Dedicated to my teacher Professor A.Blokh}
	
	\author{Sourav Bhattacharya}

	\address[Dr. Sourav Bhattacharya]
	{Department of Mathematics, Visvesvaraya National Institute Of Technology Nagpur, 
		Nagpur, Maharashtra 440010, India}
	\email{souravbhattacharya@mth.vnit.ac.in}



\begin{abstract}
	We obtain a complete characterization of \emph{topologically exact patterns} on \emph{triods}. Based on their \emph{rotation number} $\rho$, these \emph{exact patterns} are grouped into three classes: \emph{slow} ($\rho < \frac{1}{3}$), \emph{fast} ($\rho > \frac{1}{3}$) and \emph{ternary} ($\rho = \frac{1}{3}$).  For each category, we derive a \emph{linear ordering} of the set of natural numbers, $\mathbb{N}$ that captures  \emph{forcing} between the \emph{patterns}. We also show that each of these orderings is \emph{stable} under perturbations.
	
\end{abstract}

	\maketitle
		\section{Introduction}\label{intro}

		In 1964, A. N Sharkovsky in his seminal paper \cite{S}, established a groundbreaking result that completely describes all possible sets of periods of periodic orbits (also called cycles) for continuous self-maps of the interval. To state his theorem, we first recall the \emph{Sharkovsky ordering} on the set of natural numbers,  $\mathbb{N}$:
		$$3\sha 5\sha 7\sha\dots\sha 2\cdot3\sha 2\cdot5\sha 2\cdot7 \sha \dots $$
		$$
		\sha\dots 2^2\cdot3\sha 2^2\cdot5\sha 2^2\cdot7\sha\dots\sha 8\sha
		4\sha 2\sha 1$$ 
		
		For each $k \in \mathbb{N}$, define $Sh(k) = \{\, m \in \mathbb{N} : k \sha m \,\} \cup \{k\}$
		and set $Sh(2^\infty) = \{\, 1, 2, 4, 8, \dots, 2^n, \dots \,\}$. 	Let $\mathrm{Per}(f)$ denote the set of all periods of cycles of $f$. The \emph{Sharkovsky Theorem} may then be stated as follows:
		
		\begin{theorem}[\cite{shatr}]\label{t:shar}
			Let $f:[0,1] \to [0,1]$ be a continuous map. If $m,n \in \mathbb{N}$ with $m \sha n$ and $m \in \mathrm{Per}(f)$, then $n \in \mathrm{Per}(f)$. Consequently, there exists some $k \in \mathbb{N} \cup \{2^\infty\}$ such that 
			$\mathrm{Per}(f) = Sh(k)$. Conversely, for every $k \in \mathbb{N} \cup \{2^\infty\}$, there exists a continuous map $f:[0,1] \to [0,1]$ satisfying $\mathrm{Per}(f) = Sh(k)$.	
		\end{theorem}
		
		Theorem \ref{t:shar} elucidates a hidden rich combinatorial framework which controls the disposition of periodic orbits of a continuous interval map and led to the inception of a new direction of research known as \emph{combinatorial dynamics}.  Also, Theorem \ref{t:shar}  introduces the notion of a \emph{forcing relation}. Specifically, if $m \sha n$, then the existence of a cycle of period $m$ for an interval map necessarily implies the existence of a cycle of period $n$. In this way, the theorem reveals how different “\emph{types}” of cycles (where the “\emph{type}” refers to the \emph{period}) are interconnected through \emph{forcing}.

		Building on this, several avenues of research naturally emerge. One direction is to establish a more “\emph{refined}” framework describing the coexistence of cycles,  than that offered by Theorem \ref{t:shar}. Now, the finest possible classification of cycles is through their \emph{cyclic permutation}—that is, the \emph{cyclic ordering} induced by how the map permutes the points of the cycle when arranged from left to right. As it turns out, classifying \emph{cycles} in this manner is too detailed and doesn't yield a transparent picture (see \cite{Ba}).  This motivated the development of a  middle-of-the-road way of describing \emph{cycles}: \emph{rotation theory}.

		The idea of \emph{rotation numbers} originated with Poincaré in his study of circle homeomorphisms (see \cite{poi}). It was later extended to degree-one circle maps by Newhouse, Palis, and Takens \cite{npt83}, and subsequently explored in works such as \cite{bgmy80, ito81, cgt84, mis82, mis89, almm88}. (See chapters 2 and 3 of \cite{alm00} for comprehensive references.) In a broad setting, rotation numbers may be introduced as follows:
		
		\begin{definition}[\cite{mz89,zie95}]
			Let $X$ be a \emph{compact metric space} with a Borel $\sigma$-algebra, $\phi:X\to\mathbb R$ be a
			\emph{bounded measurable function} (often called an \emph{observable}) and 
			$f:X\to X$ be a \emph{continuous} map. Then for any $x \in X$ the set
			$I_{f,\phi}(x)$, of all \emph{sub-sequential limits} of the sequence
			$ \left \{ {\frac1n} \sum^{n-1}_{i=0}\phi(f^i(x)) \right \}$ is called the {\it
				$\phi$-rotation set} of $x$.  If $I_{f,\phi}(x)=\{\rho_\phi(x)\}$ is a singleton, then the
			number $\rho_\phi(x)$ is called the {\it $\phi$-rotation number} of
			$x$. 
		\end{definition}
		It is easy to see that the $\phi$-\emph{rotation set}, $I_{f,\phi}(x)$ is a closed
		interval for all $x \in X$. The union of all $\phi$-\emph{rotation sets} of all points of $X$ is called
		the \emph{$\phi$-rotation set} of the map $f$ and is denoted by
		$I_f(\phi)$.   If $x$ is an $f$-periodic point of period $n$ then its {\it $\phi$-rotation number} 
		$\rho_\phi(x)$ is well-defined, and a related concept of the
		\emph{$\phi$-rotation pair} of $x$ can be introduced: the pair
		$({\frac1n}\sum^{n-1}_{i=0}\phi(f^i(x)), n)$ is the
		\emph{$\phi$-rotation pair} of $x$. 
		
		Another important direction inspired by Theorem \ref{t:shar} is its extension to more complex spaces. In this paper, we pursue both directions simultaneously. 	A \emph{triod} $\tau$ is defined as $\tau = \{ z \in \mathbb{C}: z^{3} \in [0,1]\}$. Geometrically, it may be viewed as a continuum formed by three copies of $[0,1]$,  joined at a common endpoint, called the \emph{branching point} $``a"$. Each connected component of $\tau \setminus \{a\}$ is called a \emph{branch} of $\tau$.  The set of periods of cycles for a continuous map $f: \tau \to \tau$, where the central \emph{branching point} $a$ remains fixed, was studied in~\cite{alm98,Ba}. A more detailed account was provided in~\cite{almnew}, where it was demonstrated that the set of possible periods can be represented as union of ``initial segments'' of certain linear orderings, each associated with rational numbers in the interval $(0,1)$ having denominators not exceeding $3$. These orderings were defined on specific subsets of the rationals. However, this phenomenon was only empirically observed and lacked a theoretical proof.  Finally in 2001, Blokh and Misiurewicz (see~\cite{BMR}),  introduced \emph{rotation theory} for \emph{triods} and provided a coherent justification of the previously observed phenomenon.  

Building up on the results obtained in \cite{BMR}, the concept of \emph{triod-twists}—the \emph{simplest} cycles associated with a prescribed \emph{rotation number} $\rho$—was introduced in~\cite{BB5}. In~\cite{BB5}, such cycles were systematically studied, leading to their complete characterization. Moreover, the dynamics of all possible \emph{unimodal triod-twist} cycles corresponding to a given \emph{rotation number} were also described. In this paper, we continue the investigation of maps on $\tau$ in the framework developed in \cite{BMR} and \cite{BB5}, aiming to establish counterparts of the classical results for interval maps. 

In particular, Blokh and Misiurewicz studied  \emph{exact patterns} for interval maps in \cite{BM3} and defined a linear order on $\mathbb{N}$ that reflects the forcing relations among these patterns. We demonstrate that an analogous phenomenon persists for maps on $\tau$, although the arguments are significantly more delicate and the resulting structure is markedly more complex. We next describe our approach in greater detail.
	
  We consider the set $\mathcal{U}$ of all continuous maps of $\tau$ into itself for which the \emph{central point} $a$ of $\tau$ is the \emph{unique} fixed point. We write $ x>y$ if $x$ and $y$ lie on the same branch of $\tau$ and $x$ is farther away from $a$ than $y$; write $ x \geqslant y$ if $x>y$ or $x=y$. 	We call two \emph{cycles} $P$ and $Q$ on $\tau$ \emph{equivalent} if there exists a homeomorphism $h: [P] \to [Q]$ \emph{conjugating} $P$ and $Q$ and \emph{fixing} \emph{branches} of $\tau$. The  \emph{equivalence classes} of \emph{conjugacy} of a \emph{cycle} $P$ is called the \emph{pattern} of $P$. A cycle $P$ of a map $f \in \mathcal{U}$ is said to \emph{exhibit} a \emph{pattern} $A$ or is of \emph{pattern} $A$ or is a \emph{representative} of the \emph{pattern} $A$ in $f$ if $P$ belongs to the \emph{equivalence class} $A$. A \emph{pattern} $A$ \emph{forces} a \emph{pattern} $B$ if and only if any map $f \in \mathcal{U}$ with a \emph{cycle} of \emph{pattern} $A$ has also a \emph{cycle} of \emph{pattern} $B$. It follows (see \cite{alm98,BMR}) that if a \emph{pattern} $A$ forces a \emph{pattern} $B \neq A$, then $B$ doesn't \emph{force} $A$.  We say that a \emph{cycle} $P$ \emph{forces} a \emph{cycle} $Q$ if the \emph{pattern} \emph{exhibited} by $P$ \emph{forces} the \emph{pattern} \emph{exhibited}  by $Q$.  We call a \emph{cycle} and its \emph{pattern} \emph{primitive} if each of its points lies on a different \emph{branch} of $\tau$.

	A map $f_P: \tau \to \tau$ is called $P$-\emph{linear} for a cycle $P$ on $\tau$,  if it \emph{fixes} $a$, is \emph{affine} on every component of $[P] - (P \cup \{ a\})$ and also \emph{constant} on every component of $\tau - [P]$ where $[P]$ is the convex hull of $P$. The following result provides a particularly elegant characterization of all \emph{patterns} that are \emph{forced} by a given \emph{pattern} $A$.
	
	\begin{theorem}[\cite{alm98,BMR}]\label{forcing}
		Let $f$ be a $P$-linear map, where $P$ is a cycle of pattern $A$.  
		Then a pattern $B$ is forced by $A$ if and only if $f$ possesses a cycle $Q$ of pattern $B$.
	\end{theorem}

	Let $ f \in \mathcal{U}$ and  $P \subset \tau- \{a \}$  be finite. By an \emph{oriented graph} corresponding to $P$, we shall mean a graph $G_P$, whose vertices are elements of $P$ and arrows are defined as follows. For a $x,y \in P$, we will say that there is an \emph{arrow} from $x$ to $y$ and write $x \to y$ if there exists $ z \in \tau$ such that $ x \geqslant z$ and $ f(z) \geqslant y$. We will refer to a \emph{loop}  in the \emph{oriented graph} $G_P$ as a \emph{point} \emph{loop} in $\tau$ to distinguish them from \emph{loops} of \emph{intervals} which we define in Section \ref{preliminaries}. We call a \emph{point} \emph{loop} in $\tau$ \emph{elementary} if it passes through every \emph{vertex} of $G_P$ at most once. If $P$ is a cycle of period $n$, then the loop $\gamma : x \to f(x) \to f^2(x) \to f^3(x) \to \dots f^{n-1}(x) \to x$, $x \in P$ is called the \emph{fundamental point loop associated with} $P$.  
	
	Now, we are in a position to state the \emph{rotation theory} for \emph{triods} as introduced in \cite{BMR}. Let $ f \in \mathcal{U}$ , $P \subset \tau- \{a\}$ be finite and the \emph{oriented graph} $G_P$ given by $P$ is \emph{transitive} (that is there is a \emph{path} from every \emph{vertex} to every \emph{vertex}). If $P$ is a \emph{cycle}, it is easy to see that $G_P$ is always \emph{transitive}. Call each \emph{component} of $[P] - (P \cup \{ a\})$, a $P$-\emph{basic interval} on $\tau$. We denote the set of all \emph{arrows} of the \emph{oriented graph} $G_P$ by $A$.   
	
	In our \emph{model} of  $\tau$,   we consider $\tau$ as being \emph{embedded} into the plane with the \emph{central branching point} at the \emph{origin} and \emph{branches} being segments of straight-lines. Let us name the \emph{branches} of $\tau$ in the anticlockwise direction such that $B = \{ b_i | i = 0,1,2 \}$  (addition in the subscript of $b$  is modulo 3) is the collection of all its \emph{branches}. Let $A$ be the set of all \emph{arrows} of the \emph{oriented graph} $G_P$. We define a \emph{displacement function} $ d : A \to \mathbb{R}$  by $d(u \to v) = \frac{k}{3}$, where $ u \in b_i$ and $v \in b_j$ and $ j = i +k $ (modulo 3). For a \emph{point loop} $\Gamma$ in $G_P$,  denote by $ d(\Gamma)$ the sum of the values of the \emph{displacement} $d$ along the loop. In our model of $\tau$, this number tells us how many times we \emph{revolved} around the origin in the anticlockwise sense. Thus, $ d(\Gamma)$ is an integer. We call  $ rp(\Gamma) = (d(\Gamma), |\Gamma|)$  and $ \rho(\Gamma) = \frac{d(\Gamma)}{ |\Gamma|}$ as the \emph{rotation pair} and \emph{rotation number} of $\Gamma$  respectively where $|\Gamma|$ denotes length of $\Gamma$.  The closure of the set of \emph{rotation numbers} of all \emph{loops} of $G_P$ is called the \emph{rotation set} of $G_P$,  denoted by $L(G_P)$. By \cite{zie95}, $L(G_P)$ is equal to the smallest interval containing the \emph{rotation numbers} of all \emph{elementary loops} of $G_P$.

	Following the notations in \cite{BMR}, a \emph{rotation pair} $rp(\Gamma) = (mp, mq)$, where $p,q,m \in \mathbb{N}$ with $g.c.d(p,q)=1$, can be represented in the form $mrp(\Gamma) = (t, m), \quad \text{where } t = \frac{p}{q}$. The pair $(t, m)$ is referred to as the \emph{modified rotation pair (mrp)} of the point loop $\Gamma$.

	The \emph{rotation number}, \emph{rotation pair}, and \emph{modified rotation pair} of a cycle $P$ are defined to be those of its \emph{fundamental point loop} $\Gamma_P$.  Similarly, the corresponding quantities for a \emph{pattern} $A$ are defined as those of any cycle $P$ that \emph{exhibits} $A$.  The \emph{rotation interval forced} by a \emph{pattern} $A$ is defined as the \emph{rotation set} $L(G_P)$ of the \emph{oriented graph} $G_P$ associated with a cycle $P$ which \emph{exhibits} $A$.   Finally, we denote by $mrp(A)$,  the set of all \emph{modified rotation pairs} of \emph{patterns} that are \emph{forced} by the \emph{pattern} $A$.

	\begin{figure}[H]
		\caption{Schematic representation of \emph{modified rotation pairs} (\emph{mrp}) on the real line with attached \emph{prongs}}
		\centering
		\includegraphics[width=0.5\textwidth]{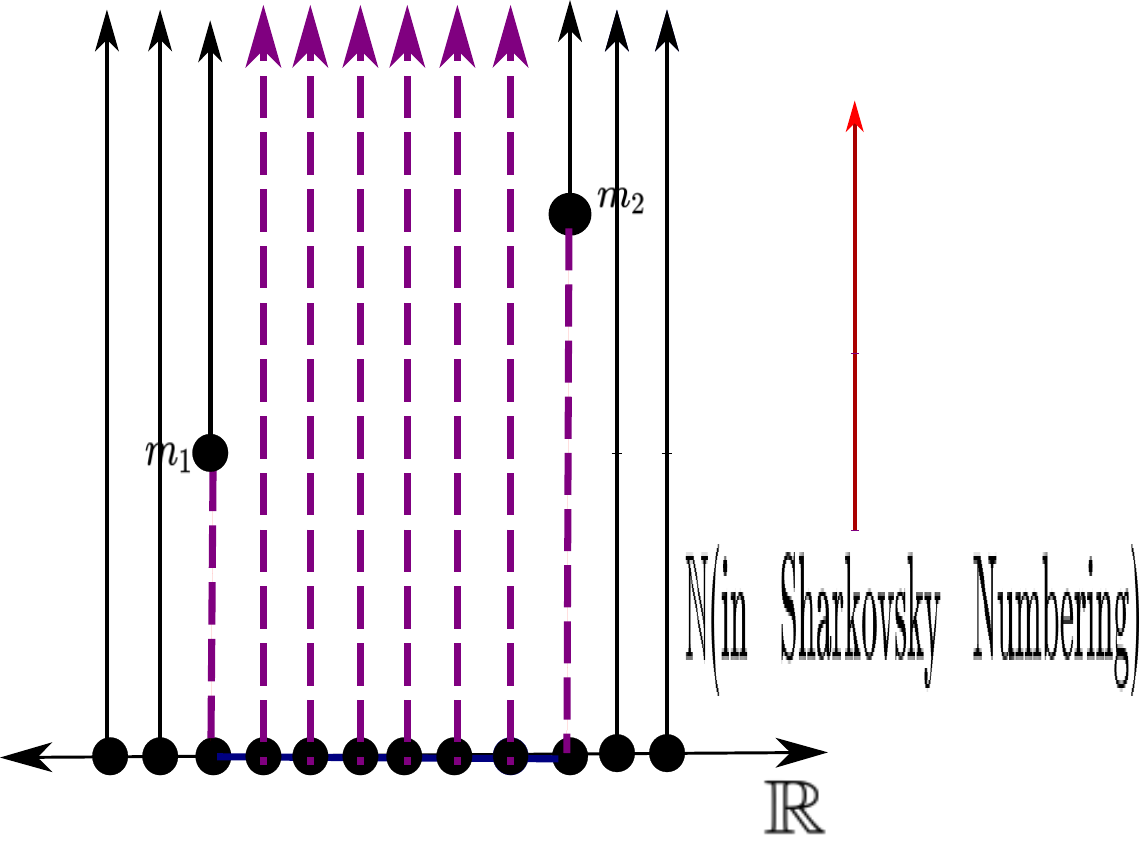}
		\label{convex_hull}
	\end{figure}

	\emph{Modified rotation pairs} admit a convenient geometric interpretation (see Figure~\ref{convex_hull}).  Consider the real line, and imagine that at each rational point a \emph{prong} is attached, while irrational points are equipped with degenerate prongs.  On the prong corresponding to each rational point, mark the set $\mathbb{N} \cup \{2^{\infty}\}$, ordered according to the \emph{Sharkovsky ordering} $\sha$, with $1$ placed nearest to the real line and $3$ placed farthest from it.  All points lying directly on the real line are labeled by $0$.  The union of the real line together with all its attached prongs will be denoted by $\mathbb{M}$.

	A \emph{modified rotation pair} $(t, m)$ is then represented by the element of $\mathbb{M}$ corresponding to the number $m$ on the prong attached at $t$.  
	No actual \emph{rotation pair}, however, corresponds to $(t, 2^{\infty})$ or to $(t, 0)$.  For two elements $(t_1, m_1)$ and $(t_2, m_2)$ in $\mathbb{M}$, the \emph{convex hull} 
	$[(t_1, m_1), (t_2, m_2)]$ is defined as the set of all \emph{modified rotation pairs} $(t, m)$ such that either $t$ lies strictly between $t_1$ and $t_2$, or $t = t_i$ and $m \in Sh(m_i)$ for $i = 1, 2$.

	\begin{definition}
		A \emph{pattern} $A$ for a map $f \in \mathcal{U}$ is called \emph{regular} if $A$ doesn't force a \emph{primitive pattern} of \emph{period} $2$; call a \emph{cycle} $P$ \emph{regular} if it \emph{exhibits} a \emph{regular pattern}. 
	\end{definition}
	
	By \emph{transitivity} of \emph{forcing}, \emph{patterns} forced by a \emph{regular pattern} are \emph{regular}. A map $f \in \mathcal{U}$ will be called \emph{regular} if all its \emph{cycles} are \emph{regular}. Let $\mathcal{R}$ be the collection of all \emph{regular} maps $ f \in \mathcal{U}$. By Theorem \ref{forcing}, if $P$ is \emph{regular}, then the $P$-\emph{linear} map $f$ is \emph{regular}.

	\begin{theorem}[\cite{BM2}]\label{result:1}

		Let $A$ be a regular pattern for a map $f \in \mathcal{U}$. Then there are patterns $B$ and $C$ with modified rotation pairs $(t_1, m_1)$ and $(t_2,m_2)$ respectively such that $mrp(A) = [(t_1, m_1), (t_2, m_2)]$. 
	\end{theorem}

	\begin{definition}
		A \emph{regular pattern} $\pi$ is called a \emph{triod twist} if it doesn't force another \emph{pattern} with the same \emph{rotation number}. 
	\end{definition}

\begin{figure}[H]
	\caption{\emph{Bifurcation diagram} illustrating the change in \emph{color} of points with varying \emph{rotation number}~$\rho$.}
	\centering
	\includegraphics[width=0.6\textwidth]{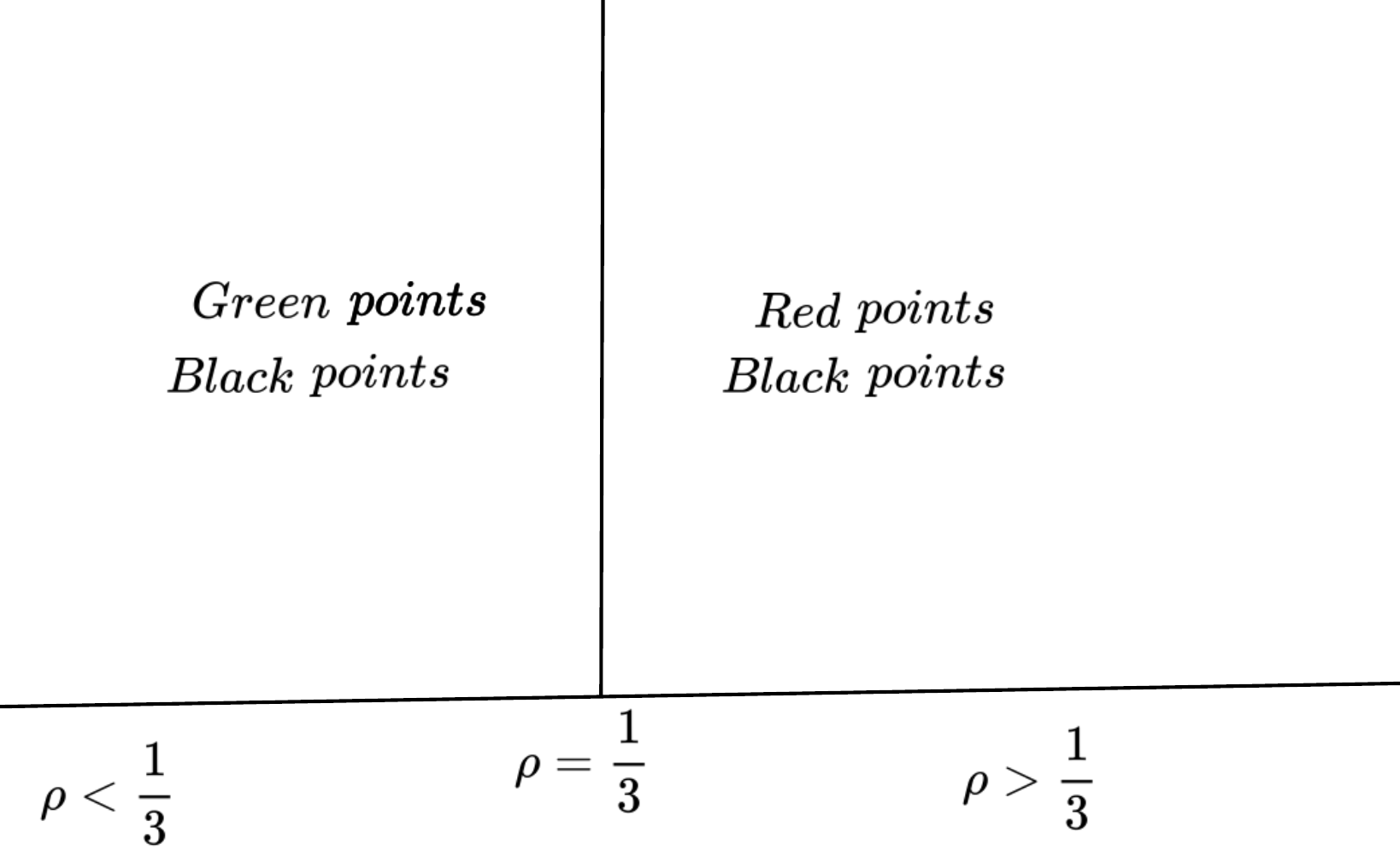}
	\label{drawing2}
\end{figure}

In describing the \emph{arrows} in the \emph{oriented graph} $G_P$ corresponding to a finite set $P \subset \tau \setminus \{a\}$, we adopt the \emph{color convention} introduced in \cite{BMR}.  
For any \emph{directed edge (arrow)} $u \to v$ in $G_P$, where $u, v \in P$, we assign the following \emph{colors} according to the \emph{displacement} value $d(u \to v)$: \emph{green} if  $d(u \to v) = 0$,  \emph{black} if  $d(u \to v) = \tfrac{1}{3}$,  \emph{red} if  $d(u \to v) = \tfrac{2}{3}$.
If $P$ is a cycle of a map $f \in \mathcal{R}$, the \emph{color} of a point $x \in P$ is defined to be the \emph{color} of the arrow $x \to f(x)$ in its \emph{fundamental point loop}.

In \cite{BB5}, a bifurcation in the qualitative behavior of \emph{triod-twist patterns}—with respect to the \emph{color} of their points—was identified at the \emph{rotation number}~$\rho = \frac{1}{3}$ (see Figure~\ref{drawing2}).

\begin{theorem}[\cite{BB5}]\label{bifurcation:one:third}\label{rot:one:third}\label{defn:triod:twist}
	Let $A$ be a triod-twist pattern with rotation number $\rho$.  
	Then:
	\begin{enumerate}
		\item If $\rho < \tfrac{1}{3}$, the pattern $A$ contains no red points.
		\item If $\rho > \tfrac{1}{3}$, the pattern $A$ contains no green points.
	\end{enumerate}
	Moreover, the pattern $\pi_3$ corresponding to the primitive cycle of period~$3$ is the unique triod-twist pattern with rotation number $\tfrac{1}{3}$.
\end{theorem}

Let us now state our plan for the paper. A continuous map $f:X\to X$ on a topological space $X$ is called \emph{topologically exact} if for every nonempty open set $U\subset X$,  there exists  $n \in \mathbb{N}$ such that $f^{n}(U)=X$.  A cycle $P$ on a \emph{triod} $\tau$ is called an \emph{exact} cycle if the $P$-\emph{linear} map, $f_P$ is an \emph{exact} map.  A \emph{pattern} $\pi$  is called \emph{exact} if any cycle $P$ which \emph{exhibits} $\pi$ is an \emph{exact cycle}.   This provides a natural dichotomy for \emph{patterns} on \emph{triods}: they are either  \emph{exact} or \emph{non-exact}.

		This paper aims to explore the problem of coexistence among the periods of \emph{exact} patterns on a \emph{triod} $\tau$. This question can be addressed once we fully characterize \emph{exact patterns}  on a \emph{triod}. To this end, we employ the notion of a \emph{block structure}, originally developed by Misiurewicz for \emph{patterns} on the  circle and the  interval by Blokh and Misiurewicz (see \cite{alm00, BM3}) and later adapted  to \emph{triod patterns} (see \cite{BMR}) by the same authors. We prove in this paper that a  \emph{pattern} $\pi$ on a \emph{triod} $\tau$ is \emph{exact} if and only if it has \emph{no block structure}. This constitutes our first main result (see Section~\ref{characterization:section}, Theorem~\ref{connection:mixing:no:block:structure}).

		Next, leveraging this result together with the frameworks developed in \cite{BMR} and \cite{BB5}, we undertake a detailed study of the structural properties of \emph{exact patterns} on \emph{triods} in Section~\ref{section:properties:mixing}. As observed in the paper \cite{BB5}, a qualitative \emph{bifurcation} occurs in the family of \emph{triod-twist patterns} at the \emph{rotation number},  $\rho = \frac{1}{3}$ (see Theorem~\ref{bifurcation:one:third}). This observation motivates a natural classification of \emph{patterns}  on \emph{triods} into three categories: those with $\rho < \frac{1}{3}$, referred to as \emph{slow patterns}; those with $\rho > \frac{1}{3}$, referred to as \emph{fast patterns}; and those with $\rho = \frac{1}{3}$, called \emph{ternary patterns}. The dynamics of \emph{exact patterns}   within these classes are analyzed in Section~\ref{section:properties:mixing} (see Section~4, Theorems~\ref{block:structure:necessary}, \ref{loop:maintainance}, \ref{green:p0:image:green}, \ref{k:n:green:forces:k+1:n+3}, and~\ref{9k+3:6k+2}).

The framework developed in Section  \ref{section:properties:mixing}  enables us in Section \ref{forcing:section}, to study the \emph{forcing relations} among \emph{slow}, \emph{fast}, and \emph{ternary} \emph{exact}  \emph{patterns} separately. This yields three distinct orderings of the natural numbers that capture the \emph{forcing structure} within these three classes. Moreover, these orderings are shown to persist under small \emph{perturbations} of the map, thereby demonstrating the robustness of the forcing structure (see Section \ref{forcing:section}, Theorems \ref{forcing:mixing:green:main:theorem}, \ref{forcing:mixing:red:main:theorem} and \ref{ternary:final}).

		The organization of the paper is as follows:
		
		\begin{enumerate}
			\item In Section~2, we state all essential definitions and theorems that will be used throughout the manuscript.
			
			\item In Section~3, we prove that a \emph{pattern} on a \emph{triod} is \emph{exact}  if and only if it has \emph{no block structure}.
			
			\item In Section~4, we investigate the structural properties of \emph{exact patterns} on \emph{triods}, providing a foundation for the subsequent analysis.
			
			\item In Section~5, we apply the results obtained in Sections~3 and~4 to derive explicit orderings among the periods of \emph{slow, fast} and \emph{ternary} \emph{exact patterns} on \emph{triods}, which depict the \emph{forcing relation} among these \emph{patterns}. We also establish the \emph{stability} of this ordering.

		\end{enumerate}

		 	\section{Preliminaries}\label{preliminaries}

\subsection{Monotonicity}\label{monotonicity}

A continuous map $f: \tau \to \tau $ is said to be \emph{monotone} on a subset $U \subset \tau$ if, for every $v \in f(U)$, the preimage $f^{-1}(v)$ is a \emph{connected subset} of $U$.  A subset $U \subset \tau$ is called a \emph{lap} of $f$ if it is a \emph{maximal open subset} of $\tau$ on which $f$ is \emph{monotone}, maximality being understood with respect to set inclusion.  The number of \emph{laps}  of a map $f \in \mathcal{U}$ is referred to as the \emph{modality} of $f$.  Similarly, the \emph{modality} of a periodic orbit (or \emph{cycle}) $P$ is defined as the \emph{modality}  of the $P$-\emph{linear} map $f$ \emph{associated} with $P$.

\subsection{P-adjusted} 
	Given a cycle $P$, a map $f \in \mathcal{U}$ is said to be \emph{$P$-adjusted} if it has no other cycle, distinct from $P$, that exhibits the same \emph{pattern} as $P$.

\begin{theorem}[\cite{alm98}]\label{P:adjusted}
	For any cycle $P$ of a map $f \in \mathcal{U}$, there exists a $P$-adjusted map $g$ such that $g$ coincides with $f$ on $P$; that is, $f|_P = g|_P$.
\end{theorem}

	\subsection{Loops}\label{loops}

In Section \ref{intro}, we introduced \emph{oriented graph} $G_P$  and \emph{point loops} corresponding to a finite set $P \subset \tau - \{ a\}$. The following result suggests that to find out the patterns forced by a given pattern $A$, it is sufficient to look at the \emph{point loops} in the \emph{oriented graph} $G_P$ where $P$ exhibits $A$. 

\begin{theorem}[\cite{BMR}] \label{loops:orbits:connection:1} The following properties holds:
	
	\begin{enumerate}
		\item For any point loop $x_0 \to x_1 \to \dots x_{m-1} \to x_0$ in $\tau$, there is a point $y \in \tau - \{a\}$ such that $f^m(y) = y$ and for every $ k=0,1,2, \dots ,m-1$, the points $x_k$ and $f^k(y)$ lie on the same branch of $\tau$.

		\item Let $f$ be a $P$-linear map for some cycle $P \neq \{a\}$. Suppose that $y \neq a$ is a periodic point of $f$ of period $q$. Then, there exists a point loop $x_0 \to x_1 \to \dots x_{q-1} \to x_0$ in $\tau$ such that $ x_i \geqslant f^i(y)$ for all $i$. 
		
	\end{enumerate}
	
\end{theorem}

We now define \emph{loops} of \emph{intervals}. For this we  borrow the standard definitions from \cite{alm00} and \cite{alm98}.

For $x,y \in \tau$ lying in the same \emph{branch}, we call the \emph{convex hull} $[x,y]$ of $x$ and $y$, an \emph{interval} on $\tau$ connecting $x$ and $y$. An \emph{interval} $I$ on $\tau$  is said to $f$-\emph{cover} an \emph{interval} $J$ on $\tau$ if $f(I) \supset J$. Then, we can speak of a \emph{chain} of \emph{intervals}  $I_0 \RA I_1 \RA \dots $ on $\tau$ if every previous \emph{interval} on $\tau$  in the chain $f$-\emph{covers} the next one. We also speak of \emph{loops of intervals} on $\tau$. Call an \emph{interval} on $\tau$ \emph{admissible}  if one of its end-points is $a$. We
call a  \emph{chain (a loop)} of \emph{admissible intervals} $I_0,I_1,\dots $ on $\tau$  an
\emph{admissible loop (chain)} on $\tau$  respectively. Result similar to Theorem \ref{loops:orbits:connection:1} can also be obtained for \emph{loops of interval} on $\tau$. The \emph{loop of intervals} $ \Gamma : [x,a]\to [f(x), a] \to [f^2(x), a] \to \dots [f^{n-1}(x), a] \to [x,a], x \in P$ is called the \emph{fundamental admissible loop of intervals} associated with $P$.

\begin{theorem}[\cite{alm98, zie95}]\label{theorem:interval:graph}
	
	For a loop of interval	$ I_0 \to I_1 \to $ $    \dots $ $  I_{q-1}  \to I_0 $ of length $q$ on $T$,  there exists a point $x_0 \in I_0$ satisfying $f^i(x_0) \in I_{i} $ for $ i \in \{0, 1, 2, $ $ \dots q-1\}$ and $f^q(x_0) = x_0$. 
\end{theorem}

\subsection{Properties of regular patterns}\label{colors} We now study properties of \emph{regular patterns} on \emph{triods}. A \emph{loop} composed entirely of \emph{black arrows} will be called a \emph{black loop}.

\begin{theorem}[\cite{BMR}]\label{black:loop:length:3}\label{no:insider}\label{all:branches}\label{canonical:numbering}\label{canonical:ordering} Let $\pi$ be a regular pattern on triods. Let $P$ be a cycle of a $P$-linear map  $f \in \mathcal{R}$ where  $P$ exhibits $\pi$. Then the following statements hold:
	\begin{enumerate}
		\item For every point $x \in P$, there exists a black loop of length $3$ passing through $x$.
		\item If $x$ is a green point, then $x > f(x)$.
		\item The cycle $P$ contains at least one point on each branch of $\tau$.
		\item The cycle $P$ always forces a primitive cycle of period~$3$.
	\end{enumerate}
	
	Moreover, there exists an ordering $\{b_i \mid i = 0, 1, 2\}$ (where indices are taken modulo~$3$) of the branches of the triod~$\tau$ such that the points $p_i \in P$, $i = 0, 1, 2$, which are closest to the branching point $a$ on each branch $b_i$, are all black.  
	This ordering is called the canonical ordering of the branches of~$\tau$.
\end{theorem}

From this point onward, we assume that the \emph{branches} of $\tau$ are arranged according to their \emph{canonical ordering}.

\subsection{Characterization of Triod-twist patterns}\label{section:triod:twist:color}

We conclude this section by providing a necessary condition for a given \emph{pattern}~$\pi$ to qualify as a \emph{triod-twist pattern}.

\begin{definition}[\cite{BB5}]\label{green}
	A \emph{regular cycle} $P$ is said to be \emph{order-preserving} if, for any two points $x, y \in P$ with $x > y$ such that $f(x)$ and $f(y)$ belong to the same \emph{branch} of~$\tau$, we have $f(x) > f(y)$.  
	A \emph{pattern}~$A$ is called \emph{order-preserving} if every cycle that \emph{exhibits} $A$ is \emph{order-preserving}.
\end{definition}

The next theorem provides a necessary condition for a \emph{pattern} to be a \emph{triod-twist pattern}.

\begin{theorem}[\cite{BB5}]\label{necesary:condition:triod:twist}
	Every triod-twist pattern is order-preserving.
\end{theorem}

\section{Characterization of regular exact patterns}\label{characterization:section}

From this point onward, all references to a \emph{pattern} or a \emph{cycle} will, unless otherwise stated, refer to a \emph{regular pattern} or a \emph{regular cycle}, respectively.  
We begin by introducing a suitable \emph{metric} on the \emph{triod}~$\tau$.

Let $b_0, b_1,$ and $b_2$ denote the three \emph{branches} of~$\tau$, with $a$ representing the \emph{central branching point}.  
For each $i = 0, 1, 2$, there exists a homeomorphism $\gamma_i : [0,1] \to b_i$ satisfying $\gamma_i(0) = a$ and $\gamma_i(1)$ being the \emph{endpoint} of~$b_i$.  
Define a function $\psi : \tau \to [0,1]$ by setting $\psi(y) = t$ whenever $\gamma_i(t) = y$ for some $i \in \{0,1,2\}$.  
In other words, $\psi(y)$ measures the normalized \emph{distance} of the point~$y$ from the \emph{central point}~$a$ along its \emph{branch}.

We now define a metric~$d_{\tau}$ on~$\tau$ by
\[
d_{\tau}(x, y) =
\begin{cases}
	|\psi(x) - \psi(y)|, & \text{if $x$ and $y$ belong to the same $branch$,} \\[6pt]
	\psi(x) + \psi(y), & \text{if $x$ and $y$ belong to different $branches$.}
\end{cases}
\]

For $f, g \in \mathcal{R}$, we define $D(f, g) =  \displaystyle \sup_{x \in \tau} d_{\tau}\big(f(x), g(x)\big)$. A \emph{neighborhood} of a map $f \in \mathcal{R}$ will henceforth refer to one taken with respect to the metric~$D$ on~$\mathcal{R}$. Furthermore, we shall call any open connected subset of the \emph{triod}~$\tau$ an \emph{open interval} on~$\tau$, and its \emph{diameter} (with respect to~$d_{\tau}$) its \emph{length}. Let $P$ be a cycle of a $P$-\emph{linear} map $f \in \mathcal{R}$.  We will call the components of the set $[P] \setminus (P \cup \{a\})$, $P$-\emph{basic intervals}.

	\begin{lemma}\label{locally:expanding}
		Let $P$ be a  cycle of a $P$-linear map $f \in \mathcal{R}$ of period $n$. Let $J$ be a $P$-basic interval. Suppose there exists $m \in \mathbb{N}$, $1 < m < n$ and an open interval  $K \subset J$  such that $f^i(K) \cap P = \varnothing$ for $i=0,1,2, \dots m$ and $f^m(K) \subset J$. Then, the length of $f^m(K)$ is strictly greater than the length of $K$. 
	\end{lemma}

\begin{proof}
	Choose $x \in K$. Consider the maximal (in terms of set inclusion) \emph{open interval} $L \subset J$ containing $x$ such that $f^i(L) \cap P = \varnothing$ for $i=0, 1,2,\dots m$. Let $\ell$ and $r$ be the left and right \emph{endpoints} of the \emph{open interval} $L$. We claim that $f^m(\ell) \in P$ and $f^m(r) \in P$. Otherwise, we can find a slightly larger \emph{open interval} $L' \supset L$ which also satisfies $f^i(L') \cap P = \varnothing$ for $i=0, 1,2,\dots m$, contradicting the maximality of $L$.  Moreover, since $f$ is $P$-\emph{linear}, the restriction $f^m|_{L}$ is \emph{monotone} on $L$. Hence $f^m(L) = J$, and the result follows.

	\end{proof}

In \cite{BMR}, Blokh and Misiurewicz generalized to maps on \emph{triods} the notion of \emph{block structure}—a concept originally introduced for maps of the circle (see \cite{alm00, BM3}).

\begin{definition}
	Let $P$ be a periodic orbit of  a $P$-\emph{linear} map $f \in \mathcal{R}$ on a \emph{triod}~$\tau$.  
	We say that $P$ possesses a \emph{block structure} over another cycle $Q$ if $P$ can be partitioned into disjoint subsets $P = P_1 \cup P_2 \cup \dots \cup P_m$,
	called \emph{blocks}, all of equal cardinality, where $m$ is the period of $Q$.  
	These blocks satisfy the following conditions:
	\begin{enumerate}
		\item the \emph{convex hulls} $[P_i]$ are pairwise disjoint and none contains the \emph{branching point}~$a$ of~$\tau$;
		\item each \emph{block}~$P_i$ contains exactly one point~$x_i$ from~$Q$;
		\item whenever $f(x_i) = x_j$, we have $f(P_i) = P_j$.
	\end{enumerate}
\end{definition}

The same terminology is applied to \emph{patterns}.  
Specifically, we say that a \emph{pattern}~$A$ has a \emph{block structure} over a \emph{pattern}~$B$ if there exists a cycle~$P$ which \emph{exhibits}~$A$ and admits a \emph{block structure} over a cycle~$Q$ which \emph{exhibits}~$B$.  Two fundamental results concerning \emph{block structures}, due to Blokh and Misiurewicz \cite{BMR}, are stated below.

\begin{theorem}[\cite{BMR}]\label{add:3}
	Let $A$, $B$, and $C$ be patterns.  
	If $A$ has a block structure over~$B$ and $A$ forces~$C$, then either $C$ also has a block structure over~$B$, or $B$ forces~$C$.  
	Moreover, if $P$ exhibits $A$ for a $P$-linear map $f \in \mathcal{R}$, then for any pattern $C$ with a block structure over~$B$, there exists a representative  $Q$ of $C$ contained in $ \displaystyle \bigcup_i [P_i]$, where $P_i$ are the blocks of~$P$.
\end{theorem}

\begin{theorem}[\cite{BMR}]\label{add:4}
	Assume that a pattern~$A$ forces a pattern~$B$ of period~$m$, where $A$ has no block structure over~$B$ and $B$ is not a doubling.  Then, for every integer $k \in \mathbb{N}$, the pattern~$A$ forces another pattern~$C_k$ of period~$km$ which has a block structure over~$B$.
\end{theorem}

\begin{lemma}\label{basic:eventually:meeting:P}
	Let $P$ be a cycle of a $P$-linear map $f \in \mathcal{R}$. Suppose $P$ has no block structure. Then every $P$-basic interval $J$ eventually covers the convex hull of $P$. 
\end{lemma}

\begin{proof}
	Consider the set $\mathcal{A} =  \displaystyle \bigcup_{i=0}^{\infty} f^i(J)$. We first claim that $\mathcal{A}$ is connected. Suppose, on the contrary, that $\mathcal{A}$ has $k>1$ connected components $A_1, A_2, \dots, A_k$. Since $f(A_i)$ is a connected subset of $\mathcal{A}$ for each $i$, there exists a permutation $\theta : \{1,2,\dots,k\} \to \{1,2,\dots,k\}$ such that $f(A_i) = A_{\theta(i)} \quad \text{for all } i$. In particular, the sets $A_i \cap P$, $i=1,2,\dots,k$, form non-trivial \emph{blocks} of $P$. This contradicts the assumption that $P$ has no \emph{block structure}. Hence, $\mathcal{A}$ is connected.

	Next, let the \emph{branches} of $\tau$ be \emph{canonically ordered}, and let $p_i$ ($i=0,1,2$) denote the point of $P$ closest to the \emph{branching point} $a$ in each \emph{branch} $b_i$. Define $B$ as the convex hull of the points $p_0, p_1$, and $p_2$.  Since $\mathcal{A}$ is connected, and the points of $P$ are cyclically permuted by $f$,  there exists $m \in \mathbb{N}$ such that $f^m(J) \supseteq B$. Moreover, it is easy to see  that $f(B) \supseteq B$. Consequently, for all $j \in \mathbb{N}$,  $j \geqslant  m$ we have $f^{j+1}(J) \supseteq f^{j}(J)$. By induction, it follows that there exists $p \in \mathbb{N}$ such that $f^p(J)$  \emph{covers} the entire convex hull of $P$.

\end{proof}

\begin{lemma}\label{open:set:eventually:meets:P}
	Let $P$ be a cycle of a $P$-linear map $f \in \mathcal{R}$. Suppose $P$ has no block structure. Then, for any open interval in $\tau$ contained in $[P]$,  there exists $m \in \mathbb{N}$ such that $f^m(U) \cap P \neq \varnothing$. 
\end{lemma}

\begin{proof}
Let us assume to the contrary that $f^i(U) \cap P = \varnothing$ for all $ i \in \mathbb{N}$. Then, there exists a $P$-\emph{basic interval} $J$ such that $f^i(U) \subset J$ for infinitely many $ i \in \mathbb{N}$. Let $M = \{ j \in \mathbb{N}   $  $ | f^j(U) \subset J\}$. Let $\Gamma = \displaystyle \bigcup_{i \in M} f^i(U)$. From Lemma \ref{locally:expanding}, if $i_1,i_2 \in M$, $i_2 >i_1$, then the length of $f^{i_2}(U)$ is larger than the length of $f^{i_1}(U)$. But the total length of the $P$-\emph{basic interval} $J$ is finite. 	So, $\Gamma$ can have at most finitely many components. So, we can choose the largest component $C$ of $\Gamma$. By assumption, there exists $\ell \in \mathbb{N}$ such that $f^{\ell}(C) \subset J$. But this means by Lemma \ref{locally:expanding} that length of $f^{\ell}(C)$ is larger than the length of $C$ which is a contradiction, since  $C$ is the largest component of $\Gamma$.
\end{proof}

We now prove the main result of this section.

\begin{theorem}\label{connection:mixing:no:block:structure}
	Let $P$ be a cycle of a $P$-linear map $f \in \mathcal{R}$. Then $f$ is exact if and only if $P$ has no block structure. 
\end{theorem}

\begin{proof}
	Assume first that $P$ has a \emph{block structure}. Then every \emph{open interval} $K$ contained in the convex hull of a \emph{block} has its forward images contained within convex hulls of \emph{blocks}. Consequently, $f$ cannot be \emph{exact} .

	Now suppose $P$ has no \emph{block structure}, and let its period be $q$. Take an  \emph{open interval} $U \subset [P]$. By Lemma~\ref{open:set:eventually:meets:P}, there exists $m \in \mathbb{N}$ such that $f^m(U)$ contains a point $x_0 \in P$. Choose a subset $K \subset f^m(U)$ lying entirely in a $P$-\emph{basic interval} $J$ with $x_0$ as an endpoint.  
	
	If either $f^q(K)$ or $f^{2q}(K)$ already \emph{covers} $J$, the claim follows immediately. Otherwise, depending on the orientation about $x_0$, one of these images lies strictly inside $J$. Without loss of generality, assume that $f^{2q}(K) \subsetneq J$; the other case is analogous. Consider the sequence of sets, $\{ f^{j \cdot 2q}(K) : j \in \mathbb{N} \}$. Each element of this sequence has $x_0$ as an endpoint. By Lemma~\ref{locally:expanding}, $f^{2q}(K)$ is strictly longer than $K$, and $f^{4q}(K)$ is strictly longer than $f^{2q}(K)$. If $f^{4q}(K)$ covers $J$, we are done; otherwise, $f^{4q}(K) \subsetneq J$, and again Lemma~\ref{locally:expanding} ensures that $f^{8q}(K)$ has strictly greater length than $f^{4q}(K)$. Since, the length of $J$ is finite,  repeating this argument,  we conclude that there exists $\ell \in \mathbb{N}$ such that $f^{\ell \cdot 2q}(K)$ \emph{covers} $J$. Since, $K \subseteq U$, so  $f^{\ell \cdot 2q}(U)$ also \emph{covers} $J$.  
	
	Finally, by Lemma~\ref{basic:eventually:meeting:P}, there exists $s \in \mathbb{N}$ with $s > \ell \cdot 2q$ such that $f^s(U) = [P]$. This proves that the map $f$ is an \emph{exact map}. 
\end{proof}

\section{Properties of Regular exact patterns}\label{section:properties:mixing}
In this section we will study general properties of \emph{regular exact patterns}. We begin by formulating a criterion that allows us to identify \emph{regular exact patterns} based on its \emph{rotation pair}.

\begin{theorem}\label{block:structure:necessary}
	Let $P$ be a cycle of a $P$-linear map $f \in \mathcal{R}$ with rotation pair $(k,m)$. If $P$ has a block structure with $q$ points in each block, then $q$ divides both $k$ and $m$. In particular, if $k$ and $m$ are coprime, then $P$ is an exact cycle.
\end{theorem}

\begin{proof}
	Suppose $P$ has  \emph{block structure}. Since $P$ has period $m$ and each \emph{block} contains $q$ points, it follows that $q$ divides $m$. Collapsing each \emph{block} to a single point yields a cycle $Q$ of period $m' = \tfrac{m}{q}$. Let the \emph{rotation pair} of $Q$ be $(k',m')$. Since $P$ and $Q$ have the same \emph{rotation number}, we have $\frac{k}{m} = \frac{k'}{m'} = \frac{k'}{\frac{m}{q}}$. This implies $k = k' q$, and hence $q$ divides $k$. Therefore, $q$ divides both $k$ and $m$, completing the proof.
\end{proof}

Next, we show that a \emph{point loop} in a \emph{triod} $\tau$ persists under sufficiently small \emph{perturbations} of the map.

\begin{theorem}\label{loop:maintainance}
	Let $f \in \mathcal{R}$. Let $\Gamma: x_0 \to x_1 \to x_2 \to \dots x_{n-1} \to x_0$ be a point loop for $f$ in $\tau$, $x_i \in \tau$, $i \in \{ 0,1, 2 , \dots n-1\}$. Then, there exists a neighborhood $N$ of $f$ in $\mathcal{R}$, such that for each $g \in N$, $\Gamma$ is a point loop for $g$ in $\tau$. 
\end{theorem}

\begin{proof}
	By definition of a \emph{point loop}, for each $x_i$, $i \in \{ 0,1, 2 , \dots n-1\}$,  there exists $z_i$ with $x_i > z_i$ and $f(z_i) > x_{i+1}$. Since these inequalities are strict, they remain valid for all maps sufficiently close to $f$ with respect to the metric $D$. Thus, there exists a neighborhood $N$ of $f$ in $\mathcal{R}$ such that for every $g \in N$, we still have $x_i > z_i$ and $g(z_i) > x_{i+1}$. Hence $\Gamma$ persists as a \emph{point loop} for $g$.
\end{proof}

From Theorem~\ref{bifurcation:one:third}, the \emph{color} associated with the points of a \emph{triod-twist pattern} $\pi$ is completely determined by its \emph{rotation number} $\rho(\pi)$. Specifically, if $\rho(\pi) < \tfrac{1}{3}$, the \emph{pattern} $\pi$ comprises only \emph{green} and \emph{black} points, whereas for $\rho(\pi) > \tfrac{1}{3}$, it consists solely of \emph{red} and \emph{black} points. In the special case $\rho(\pi) = \tfrac{1}{3}$, the \emph{pattern} $\pi$ corresponds to a \emph{primitive cycle} of period three and hence consists entirely of \emph{black} points.  \emph{Rotation number} can be thought of as a measure of the ``\emph{speed}'' of a \emph{pattern}. So, this observation naturally motivates the following classification of all \emph{regular patterns} according to their \emph{rotation numbers}.

\begin{definition}\label{green:red:definition} 
	A \emph{regular pattern} $\pi$ is:
	\begin{enumerate}
		\item \emph{slow}, if $\rho(\pi) < \tfrac{1}{3}$;
		\item \emph{fast}, if $\rho(\pi) > \tfrac{1}{3}$; and
		\item \emph{ternary}, if $\rho(\pi) = \tfrac{1}{3}$.
	\end{enumerate}
\end{definition}

Let $P$ be a cycle of a $P$-\emph{linear} map $f \in \mathcal{R}$.  
For each \emph{branch} $b_i$ $(i=0,1,2)$ of $\tau$, let $p_i$ denote the point of $P$ lying closest to the \emph{branching point} $a$.  
By Theorem \ref{canonical:ordering}, there exists a \emph{canonical ordering} of the \emph{branches} of $\tau$ such that $p_0, p_1, p_2$ are all \emph{black}.  
Throughout the rest of the paper, we shall assume that the \emph{branches} are \emph{canonically ordered} and that the points $p_0, p_1, p_2$ are chosen accordingly.

\begin{theorem}\label{green:p0:image:green}  
Let $P$ be a triod-twist cycle of a $P$-linear map $f \in \mathcal{R}$ whose period is strictly greater than $3$. Then:
\begin{enumerate}
	\item if $P$ is slow, at least one of the points $p_0, p_1, p_2$ is the image of a green point.  
	\item if $P$ is fast, at least one of the points $p_0, p_1, p_2$ is the image of a red point.  
\end{enumerate}

\end{theorem}  

\begin{proof}  Assume that $P$ is a \emph{slow triod-twist cycle} of a $P$-linear map $f \in \mathcal{R}$.  Suppose, for contradiction, that none of the points $p_0, p_1, p_2$ is the image of a \emph{green} point. Then each must be the image of a \emph{black} point.

	Define $f^{-1}(p_i) = q_i$ and $f(p_i) = r_i$ for $i=0,1,2$. Since $q_i$ is a \emph{black} point for $i=0,1,2$, it follows that  $q_1 \geqslant p_0, \quad q_0 \geqslant p_2, \quad q_2 \geqslant p_1$. 
	
	Since,  the period of $P$ is strictly greater than $3$, at least one of these inequalities is strict. Suppose, for instance, that $q_1 > p_0$. Because $p_0$ is a \emph{black point}, we have $f(p_0) = r_0 > p_1$. But then $q_1 > p_0$ while $f(p_0) = r_0 > f(q_1) = p_1$, which implies that $P$ is not \emph{order-preserving}. This contradicts the fact that $P$ is a \emph{triod-twist} cycle (see Theorem \ref{necesary:condition:triod:twist}).
	
The argument in the case where $P$ is a \emph{fast triod-twist cycle} is entirely analogous, with ``\emph{green}'' replaced by ``\emph{red}.''  
\end{proof}

\begin{theorem}\label{k:n:green:forces:k+1:n+3}\label{k:n:red:forces:k+1:n+3}
	Let $f \in \mathcal{R}$. Let $f$ has a twist cycle $P$ with rotation pair $(k, n)$ and period strictly greater than $3$. Then,
	
	\begin{enumerate}
		\item if $P$ is slow,  then there exists a neighborhood $N$ of $f$ such that for each $g \in N$, $g$ has  an exact slow cycle with rotation pair $(k+1,n+3)$.

		\item if $P$ is fast, then there exists a neighborhood $N$ of $f$ such that for each $g \in N$, $g$ has  an exact fast cycle with rotation pair $(k+1,n+3)$. 
	\end{enumerate}

\end{theorem}

\begin{proof}

	 (i) Assume first that $P$ is a \emph{slow triod-twist cycle} of a $P$-linear map $f \in \mathcal{R}$. By Theorem \ref{green:p0:image:green}, one of $p_0, p_1$ and $p_2$ is the image of a \emph{green} point. Without loss of generality assume $p_0$ to be the image of a \emph{green} point. Let $\alpha$ be the \emph{fundamental admissible loop} of intervals associated with $P$. We replace $p_0$ in the \emph{point loop} $\alpha$ by $p_0 \to p_1 \to p_2 \to p_0$ to form a new \emph{ point loop} $\beta$ under $f$. Clearly, the \emph{rotation pair} of $\beta$ is $(k+1, n+3)$. By Theorem \ref{loop:maintainance}, there exists a neighborhood $N$ of $f$ such that for each $g \in N$, $g$ has the \emph{point loop} $\beta$. By Theorem \ref{loops:orbits:connection:1}, for each $g \in N$, $g$ has a cycle   $Q_g$ with \emph{rotation pair} $(k+1, n+3)$, \emph{associated} with the \emph{point loop} $\beta$.

	We now show that $Q_g$ has \emph{no block structure}. Let $f^{-1}(p_0) = q_0$. Then, $q_0 > p_0$ by assumption. Consider the segment $q_0 \to p_0 \to p_1 \to p_2 \to p_0$ of the \emph{ point loop} $\beta$. Again, by Theorem \ref{loops:orbits:connection:1},  there exists $x \in Q_g$ such that, $q_0 >x$,  $p_0 >f(x)$,  $p_1 > f^2(x)$, $p_2 > f^3(x)$ and $p_0 > f^4(x)$ (See Figure \ref{loop_k_n}).  Let $A$ be the \emph{convex hull} of the points $p_0,p_1$ and $p_2$. Clearly, $A = [p_0,a] \cup [p_1,a] \cup [p_2,a]$. It is easy to see that $f(A) \supset A$, that is, the fixed point $a$ is \emph{repelling}. So, $f^4(x) > f(x)$. Thus, $x> f^4(x) > f(x)$ in the \emph{branch} $b_0$. 
	
	\begin{figure}[H]
		\caption{Formation of \emph{point loop} of \emph{rotation pair} $(k+1,n+3)$ from the \emph{point loop} of \emph{rotation pair} $(k,n)$ in the case, $\frac{k}{n}< \frac{1}{3}$  (See Theorem \ref{k:n:green:forces:k+1:n+3})}
		\centering
		\includegraphics[width=0.5 \textwidth]{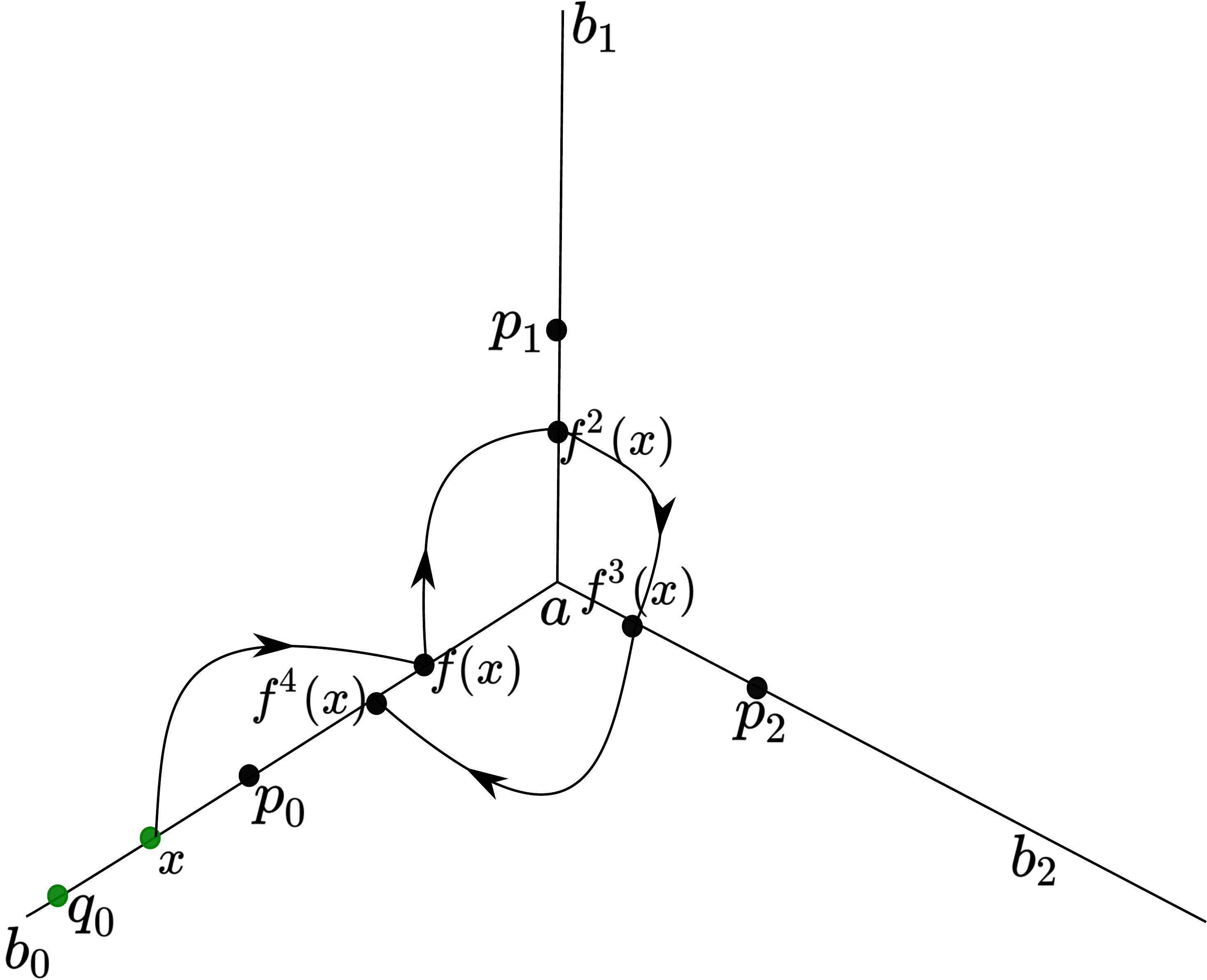}
		\label{loop_k_n}
	\end{figure}

By the construction of the \emph{loop} $\beta$, it follows that each of the intervals $[p_0,a]$, $[p_1,a]$, and $[p_2,a]$ contains exactly two points of $Q_g$, while the interval $[q_0,p_0]$ contains precisely one point of $Q_g$. Consequently, the points $x$, $f^4(x)$, and $f(x)$ form a sequence of three consecutive elements of $Q_g$ lying within the \emph{branch} $b_0$ of $\tau$. In fact, these constitute the first three consecutive points of $Q_g$ in $b_0$, measured in the direction away from the fixed point $a$.

If $Q_g$ were to possess a \emph{block structure}, then either all three points $x$, $f^4(x)$, and $f(x)$, or at least the two points $f^4(x)$ and $f(x)$ that are closest to $a$, would necessarily belong to the same \emph{block}. Suppose first that $x$ and $f^4(x)$ lie in the same block. Then their images under $f$ must also lie in the same block. However, by construction, $f(x)$ and $f^5(x)$ lie on distinct \emph{branches}, yielding a contradiction. 

Alternatively, assume that $f^4(x)$ and $f(x)$ lie in the same block. In that case, their respective pre-images must also belong to the same block. Yet, by construction, $f^3(x)$ and $x$ lie on different branches, leading again to a contradiction. 

Hence, our assumption that $Q_g$ admits a \emph{block structure} is false. Therefore, $Q_g$ has \emph{no block structure}, and the desired conclusion follows.

	(ii) The argument in the case where $P$ is a \emph{fast triod-twist cycle} is entirely analogous, with ``\emph{green}'' replaced by ``\emph{red}.''
	
\end{proof}

To prove the next result, we require the description of all possible  \emph{unimodal slow} and \emph{fast triod-twist patterns}, given in the following two theorems from the paper \cite{BB5} (See Figures \ref{rho_less_third_twist_example} and \ref{rho_greater_third_twist_example}).

\begin{figure}[H]
	\caption{The \emph{unimodal slow twist pattern} $\Gamma_0^{\frac{2}{9}}$}
	\centering
	\includegraphics[width=0.5 \textwidth]{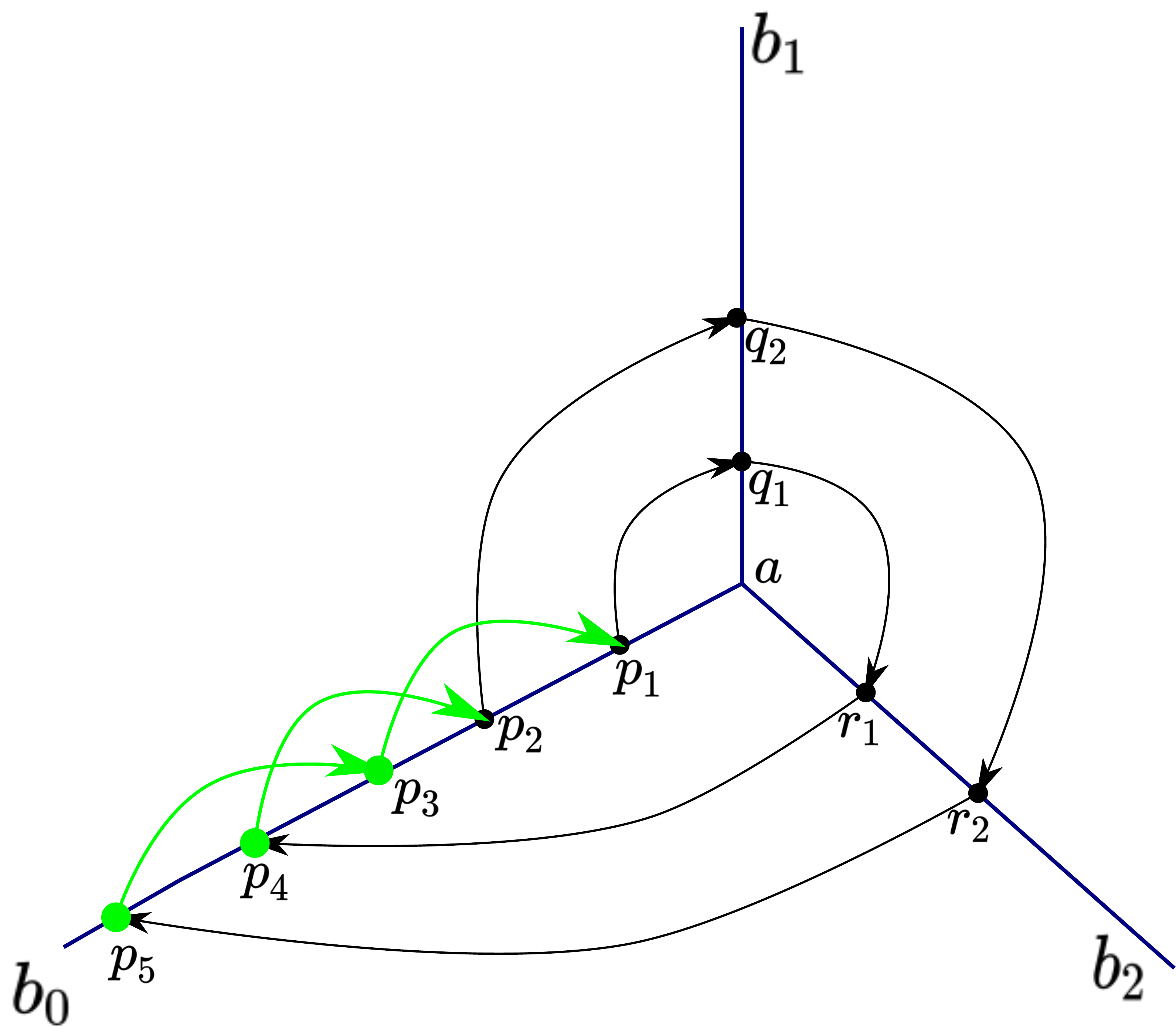}
	\label{rho_less_third_twist_example}
\end{figure}

\begin{theorem}[\cite{BB5}]\label{s:1}
	Let $m, n \in \mathbb{N}$ satisfying $g.c.d(m,n)=1$ and  $\frac{m}{n} < \frac{1}{3}$.  
	Then there exist three distinct unimodal slow triod-twist patterns $\Lambda^{\frac{m}{n}}_k,  k \in \{0,1,2\}$, each with rotation number $\frac{m}{n}$.  If $R_k^{\frac{m}{n}}$ denotes a periodic orbit exhibiting  the pattern $\Lambda^{\frac{m}{n}}_k$, its dynamics can be described as follows (See Figure \ref{rho_less_third_twist_example}):
	
	\begin{enumerate}
		\item On the branch $b_k$, there are $n-2m$ points $p_1, p_2, \dots, p_{n-2m}$,
		numbered away from the central branching point $a$, that is,  $p_{i+1} > p_i$ for all $i \in \{1,2,3, \dots n-2m\}$.  The first $m$ points $p_1, \dots, p_m$ are black, while the remaining $n-3m$ points $p_{m+1} $ $ , \dots, p_{n-2m}$ are green.  The next branch, $b_{k+1}$, contains $m$ black points $q_1, q_2,  $ $ \dots, q_m$ $(q_{j+1} > q_j$ for all $j \in \{1,2,3, \dots m\})$,
		and the third branch, $b_{k+2}$, also contains $m$ black points $r_1, r_2, \dots, r_m$ $(r_{j+1} > r_j$ for all $j \in \{1,2,3, \dots m\})$, each indexed in the direction away from $a$.
		
		\item For indices $i \in \{m+1, \dots, n-2m\}$, we have $f(p_i) = f(p_{i+m})$,
		i.e., the last $n-3m$ green points on $b_k$ are shifted by $m$ positions along the branch. The first $m$ points on $b_k$ map in an order-preserving fashion to the $m$ points of $b_{k+1}$: $f(p_i) = q_i, \quad i = 1, \dots, m$. Each point of $b_{k+1}$ maps in an order-preserving way to the corresponding point on $b_{k+2}$: $f(q_i) = r_i, \quad i = 1, \dots, m$. Finally, the $m$ points of $b_{k+2}$ map back to those of $b_k$ in an order-preserving manner: $f(r_i) = p_i, \quad i = 1, \dots, m$.
	\end{enumerate}
\end{theorem}

\begin{figure}[H]
	\caption{The \emph{unimodal fast twist pattern} $\Delta_0^{\frac{2}{5}}$}
	\centering
	\includegraphics[width=0.5 \textwidth]{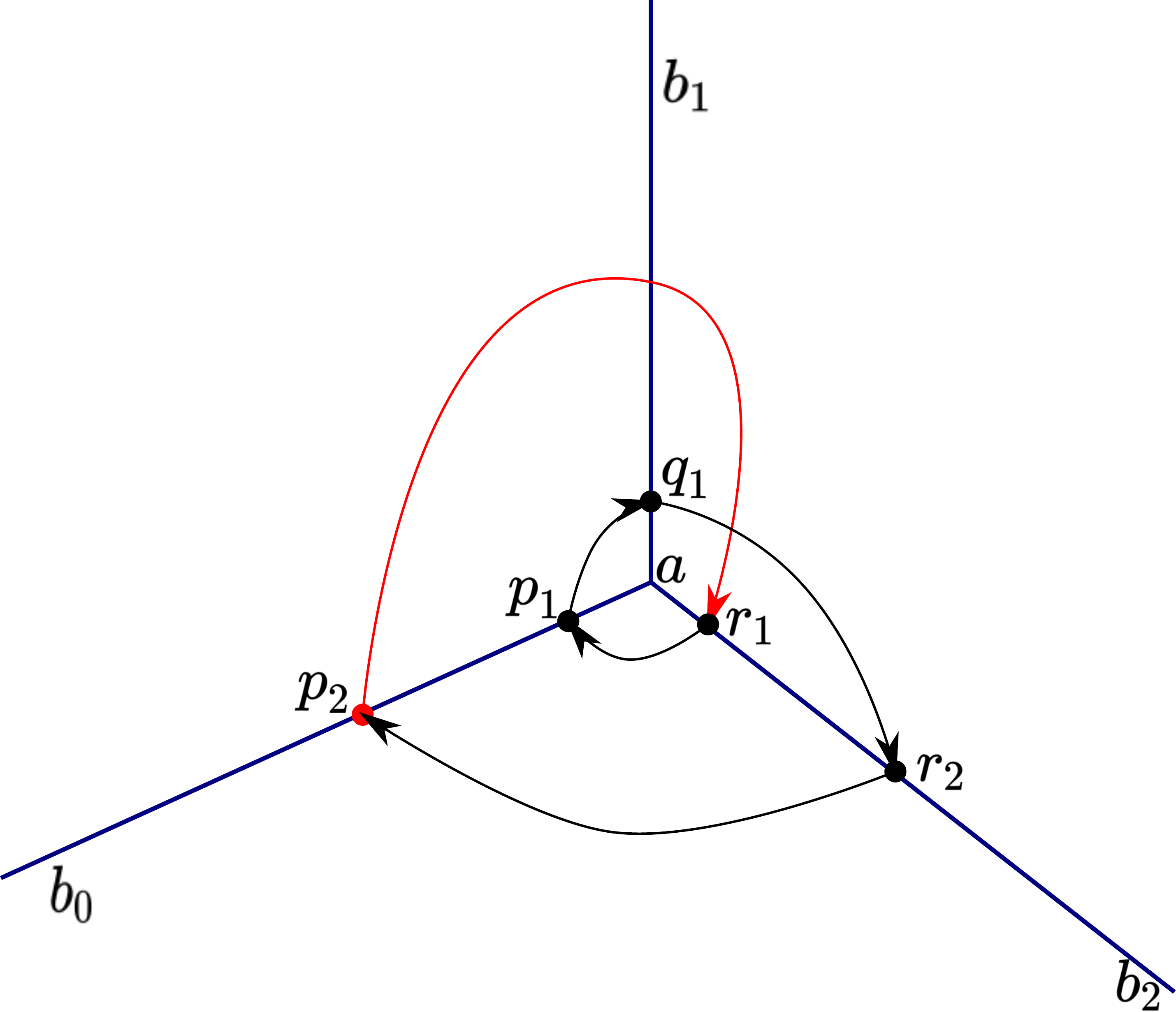}
	\label{rho_greater_third_twist_example}
\end{figure}

\begin{theorem}[\cite{BB5}]\label{s:2}
	Let $u,v \in \mathbb{N}$ with $g.c.d(u,v)=1$ and  $\frac{u}{v} > \tfrac{1}{3}$.  
	Then there exist three distinct unimodal fast triod-twist patterns $\Psi_k^{\tfrac{u}{v}}, k \in \{0,1,2\}$, each with rotation number $\tfrac{u}{v}$.  If $S_k^{\tfrac{u}{v}}$ denotes a periodic orbit exhibiting the pattern $\Psi_k^{\tfrac{u}{v}}$, the dynamics are as follows (See Figure \ref{rho_greater_third_twist_example}):
	
	\begin{enumerate}
		\item The branch $b_k$ contains $u$ points	$p_1, p_2, \dots, p_u$, numbered away from $a$.  
		The first $v - 2u$ points $p_1, \dots, p_{v-2u}$ are black, and the remaining $3u - v$ points $p_{v-2u+1}, \dots, p_u$ are red. The branch $b_{k+1}$ consists of $v - 2u$ black points $q_1, q_2, \dots, q_{v-2u}$.  The branch $b_{k+2}$ contains $u$ black points $r_1, r_2, \dots, r_u$, each indexed in the direction away from $a$.
		
		\item The last $3u - v$ red points on $b_k$ map to the first $3u - v$ points of $b_{k+2}$: $f(p_{v-2u+i}) = r_i, i = 1, \dots, 3u - v$, in an order-preserving manner. The first $v - 2u$ black points on $b_k$ map to the $v - 2u$ black points on $b_{k+1}$: $f(p_i) = q_i, \quad i = 1, \dots, v - 2u$.  The $v - 2u$ black points on $b_{k+1}$ map to the last $v - 2u$ points on $b_{k+2}$: $f(q_i) = r_{3u-v+i},  i = 1, \dots, v - 2u$, again preserving order.  Finally, each point on $b_{k+2}$ maps back to its corresponding point on $b_k$ in an order-preserving way: $f(r_i) = p_i,  i = 1, \dots, u$.
	\end{enumerate}
\end{theorem}

We now prove an important result that will be utilized in Section~\ref{forcing:section}.

\begin{theorem}\label{9k+3:6k+2}\label{9k-3:6k-2}
	
Let $P$ be an exact cycle of  a $P$-linear map $f \in \mathcal{R}$. Then, 
	\begin{enumerate}
		\item  if $P$ is slow and has period $9k+3$, for some $k \in \mathbb{N}$, then there exists a neighborhood $N$ of $f$ such that, for every $g \in N$, the map $g$ has a slow exact cycle $Q_g$ of period $6k+2$.
		\item if $P$ is fast and has  period $9k-3$,  for some $k \in \mathbb{N}$, then there exists a neighborhood $N$ of $f$ such that, for every $g \in N$, the map $g$ has a fast exact cycle $Q_g$ of period $6k-2$.
	\end{enumerate}

\end{theorem}

\begin{proof}

(1)	\quad Assume first that $P$ is an \emph{exact slow} cycle of period $9k+3$ for a $P$-\emph{linear} map $f \in \mathcal{R}$. Since the largest fraction with denominator $9k+3$, that is strictly less than $\tfrac{1}{3}$ is $\tfrac{3k}{9k+3}$, it follows from Theorem \ref{result:1} that the \emph{rotation pair} of $P$ must be $(\rho,9k+3)$ with $\rho \leqslant 3k$.  We consider two cases here:

	Case 1: $\rho < 3k$. \quad  In this case,  by Theorem~\ref{result:1}, the map $f$ has  a \emph{slow} cycle $P$ with \emph{rotation pair} $(3k-1,9k+3)$. Note that  for $k \geqslant 1,  \frac{3k-1}{9k+3} < \frac{2k-1}{6k-1} < \tfrac{1}{3}$. 	Hence, again by Theorem~\ref{result:1}, $f$ possesses a \emph{slow} cycle $Q$ with rotation pair $(2k-1,6k-1)$. Applying Theorem~\ref{k:n:green:forces:k+1:n+3}, we get that  there exists a neighborhood $N$ of $f$ such that every $g \in N$ has  an \emph{exact slow} cycle $R$ with \emph{rotation pair} $(2k,6k+2)$ and hence the result follows. 
		
	Case 2: $\rho = 3k$. \quad  In this case, the  \emph{modified rotation pair} of $P$ is  $\left(\tfrac{k}{3k+1},3\right)$ (See Figure \ref{mixing_slow_3k_9k_plus_3}). By Theorem~\ref{result:1}, the map $f$ has a cycle $R$ with \emph{modified rotation pair}  $\left(\tfrac{k}{3k+1},1 \right)$  and hence \emph{rotation pair} $(k,3k+1)$. Since $P$ is \emph{exact} , Theorem~\ref{connection:mixing:no:block:structure} implies that $P$ cannot possess a \emph{block structure} over $R$. Consequently, by Theorem~\ref{add:4}, $f$ has a cycle $Q$ with \emph{rotation pair} $(2k,6k+2)$, which has a \emph{block structure} over $R$.

		Now,  $R$ consists of a single \emph{green} point together with $3k$ \emph{black} points. Hence, $R$ is \emph{unimodal}, and its dynamics is governed by Theorem~\ref{s:1}. Since $Q$ has a \emph{block structure} over $R$, it necessarily contains $2$ \emph{green} points and $6k$ \emph{black} points (See Figure \ref{mixing_slow_3k_9k_plus_3} for the case $k=1$). The dynamics of $Q$ can likewise be determined directly from Theorem~\ref{s:1}. The restriction $f|_Q$ is \emph{monotone} on each \emph{block}, except for exactly one. For ease of exposition, we assume that $f|_Q$ is \emph{monotone} on the \emph{block} containing the \emph{green} points; the remaining cases can be treated by analogous arguments.

		Let $y_0$ denote the point of $Q$ that lies farthest from $a$ within the \emph{branch} of $\tau$ containing the two \emph{green} points of $Q$. By Theorem~\ref{s:1} and the definition of \emph{block structure}, we have $f^{3k+1}(y_0)$ equal to the second \emph{green} point of $Q$, which implies that $y_0 > f^{3k+1}(y_0)$. Moreover, the points $f^i(y_0)$ for $i \in {1,2,\dots,3k}$ and $i \in {3k+2,\dots,6k+1}$ are all \emph{black}. Since $P$ \emph{forces}  $Q$, Theorem~\ref{add:3} guarantees the existence of a point $x_0 \in P$ such that $x_0 > y_0$, and $f^i(x_0) > f^i(y_0)$ for $i \in \{0,1,2, \dots 6k+1\}$, and furthermore, $f^{6k+2}(x_0)  > y_0 = f^{6k+2}(y_0)$.

		Now, let $\Gamma$ denote the \emph{fundamental loop} of \emph{intervals associated} with $P$. Partition $\Gamma$ into three consecutive \emph{segments}: $\gamma_1 : [x_0,a] \to [f(x_0), $ $ a] \to \dots \to [f^{3k}(x_0),a], $ $ \gamma_2 : [f^{3k+1}(x_0),a] \to [f^{3k+2}(x_0),a] \to \dots \to $ $ [f^{6k+1}(x_0),a], $ and $\gamma_3 : [f^{6k+2}(x_0),a] \to [f^{6k+3}(x_0),a] \to \dots \to [f^{9k+2}(x_0),$ $ a]$.  Each \emph{segment} has length $3k+1$.  To form a \emph{exact}  cycle of the desired \emph{rotation pair}, we amalgamate $\gamma_3$ with $\gamma_2$.  Observe that $x_0$ and $f^{3k+1}(x_0)$ are the \emph{green} points of $P$, with $x_0 > f^{3k+1}(x_0)$. Since $f$ is monotone on their \emph{block}, we have $f(x_0) > f^{3k+2}(x_0)$. This allows us to construct the \emph{loop} of \emph{intervals}: $\Delta : [x_0,a] \to [f^{3k+2}(x_0),a] \to \dots \to [f^{9k+2}(x_0),a] \to [x_0,a]$, fusing $\gamma_2$ and $\gamma_3$ together.  
		
		Let $S$ be the cycle corresponding to $\Delta$, as guaranteed by Theorem \ref{theorem:interval:graph}. By construction,  $S$ has \emph{rotation pair} $(2k, 6k+2)$ and no \emph{block structure} and hence by Theorem \ref{connection:mixing:no:block:structure},  $P$ is an \emph{exact}  cycle. By Theorem \ref{loops:orbits:connection:1}, let $\delta$ denote the \emph{point loop} associated with $S$. It follows from Theorem \ref{loop:maintainance} that there exists a neighborhood $N'$ of $f$ such that, for each $g \in N'$, the map $g$ has the \emph{point loop} $\delta$. Finally, Theorem \ref{loops:orbits:connection:1} ensures that, for each such $g$, there exists an \emph{exact slow} cycle $Q_g$ with \emph{rotation pair} $(2k, 6k+2)$ \emph{corresponding} to $\delta$. This completes the proof.

	(2) \quad The argument for part (2) proceeds analogously, with the roles of ``\emph{green}" replaced by ``\emph{red}" (together with the corresponding modifications), and is therefore left to the reader. 
	
		\begin{figure}[H]
		\caption{An \emph{exact slow} cycle of \emph{rotation pair} $(3k, 9k+3)$ with $k=1$}
		\centering
		\includegraphics[width=0.6 \textwidth]{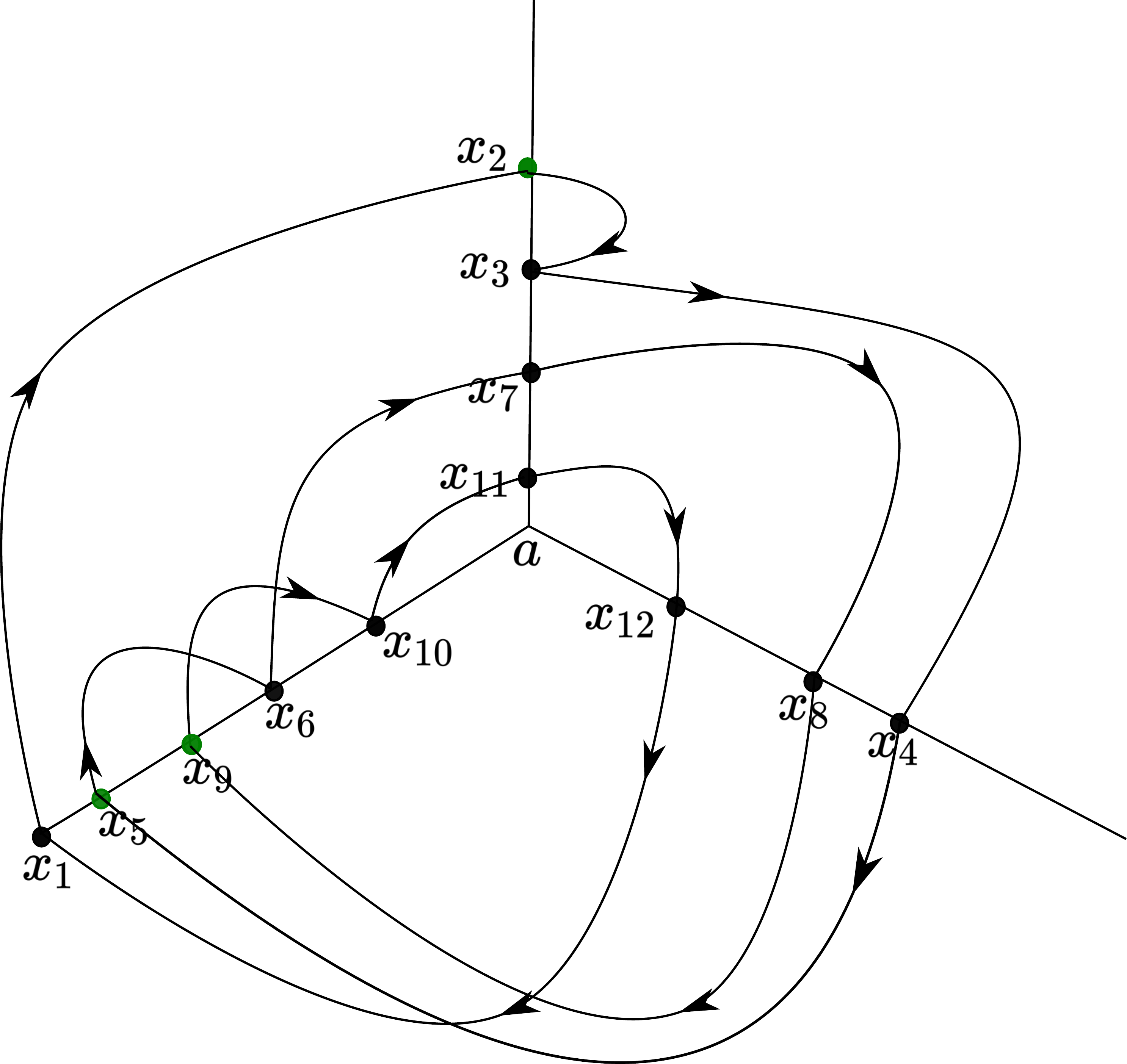}
		\label{mixing_slow_3k_9k_plus_3}
	\end{figure}

\end{proof}

\section{Forcing among regular exact patterns}\label{forcing:section}

We now apply the results established in Sections \ref{characterization:section} and \ref{section:properties:mixing} to address the problem of \emph{co-existence} among periods of \emph{exact regular patterns} on the \emph{triod} $\tau$.

\subsection{Forcing among exact slow  patterns}

We begin with \emph{exact slow patterns}. We introduce the following \emph{ordering}.

\begin{definition}\label{green:ordering}

	Let $(\mathbb{N}_{>3}, \gg_s)$ denote the partially ordered set of natural numbers greater than three,  where the order relation $\gg_s$  is defined as follows:  $6 \;\gg_s\;  5 \;\gg_s\;  11 \;\gg_s\; 17 \;\gg_s\;  23 \;\gg_s\; 29\;\gg_s\;  35\;\gg_s\;  41\;\gg_s\; 47\;\gg_s\; 53\;\gg_s\; 9\;\gg_s\; 18\;\gg_s\; 27\;\gg_s\; 36\;\gg_s\; 45\;\gg_s\; 12\;\gg_s\; 21\;\gg_s\; 30\;\gg_s\; 39\;\gg_s\; 48  \;\gg_s\; \dots$. 	The ordering begins with $6$, and thereafter the numbers are arranged according to order: 	$6k-1 \;\gg_s\; 9k \;\gg_s\; 9k+3 \;\gg_s\; 6k+2 \;\gg_s\; 3k+1 \;\gg_s\; 9k+6, k \in \mathbb{N}$. We will call the ordering $\gg_s$,  the \emph{slow ordering} of $\mathbb{N}_{>3}$. 
	
\smallskip
For each $m \in \mathbb{N}_{>3}$, define
\[
\mathcal{M}_s(m) = \{ n \in \mathbb{N}_{>3} : n \gg_s m \} \cup \{ m \}.
\]

\end{definition}

\begin{theorem}\label{forcing:mixing:green:main:theorem}
	Let $f \in \mathcal{R}$, and let $\mathcal{P}_s(f)$ denote the set of all periods associated with exact slow cycles of $f$. 
	
	If $m, n \in \mathbb{N}_{>3}$ satisfy $m \gg_s n$ and $m \in \mathcal{P}_s(f)$, then there exists a neighborhood $N$ of $f$ in $\mathcal{R}$ such that for every $g \in N$, one has $n \in \mathcal{P}_s(g)$.
	
	Consequently, there exists an integer $m \in \mathbb{N}_{>3}$ for which $\mathcal{P}_s(f) = \mathcal{M}_s(m)$,
	and for all $g \in \mathcal{R}$ sufficiently close to $f$, $\mathcal{M}_s(m) \subseteq \mathcal{P}_s(g)$.
\end{theorem}

\begin{proof}
	Observe that any natural number strictly greater than $3$ is one of the following forms: $ 6k-1, \; 9k, \; 9k+3, \; 6k+2, \; 3k+1, \;9k+6$,  for $k \in \mathbb{N}$. From Definition \ref{green:ordering}, $ 6k-1 \;\gg_s\; 9k \;\gg_s\; 9k+3 \;\gg_s\; 6k+2 \;\gg_s\; 3k+1 \;\gg_s\; 9k+6 \;\gg_s\; 6(k+1)-1 =6k+5  $,  for each $ k \in \mathbb{N}$. The largest possible \emph{rotation numbers} smaller than $  \frac{1}{3}$ for \emph{patterns} of those periods are respectively, $ \frac{2k-1}{6k-1}, \frac{3k-1}{9k}, \frac{3k}{9k+3}, \frac{2k}{6k+2}, \frac{k}{3k+1}, \frac{3k+1}{9k+6}$. These numbers are ordered as follows: $\dots <  \frac{2k-1}{6k-1} < \frac{3k-1}{9k} < \frac{3k}{9k+3} = \frac{2k}{6k+2} = \frac{k}{3k+1} < \frac{3k+1}{9k+6} < \dots \frac{1}{3}$.

	Let $6k-1 \in \mathcal{P}_s(f)$. By Theorem \ref{result:1}, $f$ has a  \emph{slow} cycle $P^1$ with \emph{rotation pair} $(2k-1,6k-1)$.  Since,  $ \frac{2k-1}{6k-1} < \frac{3k-1}{9k}$, Theorem \ref{result:1} guarantees the existence of a cycle $Q^1$ of $f$ with \emph{rotation pair} $(3k-1,9k)$.  
	By Theorem \ref{loops:orbits:connection:1}, there exists a \emph{point loop} $\gamma^1$ \emph{associated} with $Q^1$.  Theorem \ref{loop:maintainance} ensures the persistence of $\gamma^1$ in a neighborhood $N^1$ of $f$, so that for each $g \in N^1$,  there is a cycle $Q_g^1$ corresponding to $\gamma^1$, also with \emph{rotation pair} $(3k-1,9k)$.  Since,  $\gcd(3k-1,9k)=1$, Theorem \ref{block:structure:necessary} implies that $Q_g^1$ has \emph{no block structure}.  	Consequently, by Theorem \ref{connection:mixing:no:block:structure}, $Q_g^1$ is a \emph{exact}  cycle and hence $9k \in \mathcal{P}_s(g)$ for every $g \in N^1$.

	Now, suppose $9k \in \mathcal{P}_s(f)$. By Theorem \ref{result:1}, $f$ has a  \emph{slow} cycle $P^2$ with \emph{rotation pair} $(3k-1,9k)$.   Since, $3k-1$ and $9k$ are co-prime, $f$ must have a \emph{twist slow} cycle $Q^2$  with \emph{rotation pair} $(3k-1,9k)$. So, by Theorem \ref{k:n:green:forces:k+1:n+3}, there exists a \emph{neighborhood} $N^2$ of $f$, such that for each $g \in N^2$, $g$ has a \emph{exact slow} cycle $Q_g^2$ with \emph{rotation pair} $(3k,9k+3)$. Thus, $9k+3 \in \mathcal{P}_s(g)$ for every $g \in N^2$.

	If $9k+3 \in \mathcal{P}_s(f)$, Theorem \ref{9k+3:6k+2}, guarantees the existence of a neighborhood $N^3$ of $f$ such that for each $g \in N^3$, $6k+2 \in \mathcal{P}_s(g)$.

		If $6k+2 \in \mathcal{P}_s(f)$, by Theorem \ref{result:1}, $f$ has a  \emph{slow} cycle $P^4$ with \emph{rotation pair} $(2k,6k+2)$. The \emph{modified rotation pair} associated with $P^4$ is $(\frac{k}{3k+1}, 2)$. By Theorem \ref{result:1}, the map $f$ has a cycle $Q^4$ with \emph{modified rotation pair} $(\frac{k}{3k+1}, 1)$ and hence \emph{rotation pair} $(k,3k+1)$.  Then, like before from Theorems \ref{loops:orbits:connection:1} and \ref{loop:maintainance},  there exists a neighborhood $N^4$ of $f$, such that,  each $g \in N^4$ has a cycle $Q_g^4$ with \emph{rotation pair} $(k, 3k+1)$.  Since $k$ and $3k+1$ are relatively prime, Theorem \ref{block:structure:necessary} ensures that $Q_g^4$ has no \emph{block structure}. Hence, by Theorem \ref{connection:mixing:no:block:structure}, $Q_g^4$ is a \emph{exact}  cycle. Thus, $3k+1 \in \mathcal{P}_s(g)$ for each $g \in N^4$.

		Suppose,  $3k+1  \in \mathcal{P}_s(f)$. By Theorem \ref{result:1}, $f$ has a  \emph{slow} cycle $P^5$ with \emph{rotation pair} $(3k,3k+1)$.  Since, $\frac{k}{3k+1} < \frac{3k+1}{9k+6}$, Theorem~\ref{result:1} guarantees that $f$ possesses a cycle $Q^5$ with \emph{rotation pair} $(3k+1, 9k+6)$. Because $3k+1$ and $9k+6$ are relatively prime, the standard arguments apply: Theorems~\ref{loops:orbits:connection:1}, \ref{loop:maintainance}, \ref{block:structure:necessary}, and \ref{connection:mixing:no:block:structure} together guarantee a neighborhood $N^5$ of $f$ in which, for every $g \in N^5$, one has $9k+6 \in \mathcal{P}_s(g)$.

		Finally, assume $9k+6 \in \mathcal{P}_s(f)$.
		Then, by Theorem~\ref{result:1}, $f$ has a \emph{slow orbit} $P^6$ with rotation pair $(3k+1,9k+6)$.
		Since $\tfrac{3k+1}{9k+6} < \tfrac{2k+1}{6k+5}$, $f$ must also have an orbit $Q^6$ with \emph{rotation pair} $(2k+1,6k+5)$.
	As $2k+1$ and $6k+5$ are relatively prime, like before, Theorems~\ref{loops:orbits:connection:1}, \ref{loop:maintainance}, \ref{block:structure:necessary}, and \ref{connection:mixing:no:block:structure} together imply that there exists a neighborhood $N^6$ of $f$ such that, for every $g \in N^6$, one has $6k+5 \in \mathcal{P}_s(g)$.

			Now, from the \emph{transitivity} of \emph{forcing relation},  the result follows. 
\end{proof}

\subsection{Forcing among exact fast patterns}

We now investigate the \emph{forcing} relations among \emph{exact fast patterns}.

\begin{definition}\label{red:ordering}
Let $(\mathbb{N}_{>3}, \gg_f)$ denote the partially ordered set of natural numbers greater than three,  where the order relation $\gg_f$  is defined as follows: $ 6 \; \gg_f 4 \;   \gg_f 9 \; \gg_f 7 $ $ \; \gg_f 12 \; \gg_f 15 \; \gg_f 10 \; \gg_f 5\;  \gg_f $ $ 18 \; \gg_f 13 \; \gg_f 21 \; \gg_f 24  \;\gg_f 16 \; \gg_f 8 \; \gg_f 27\; \gg_f 19  \;\gg_f 30 \;\gg_f 33  \; \gg_f 22 \; \gg_f 11  \; \gg_f 36 \; \gg_f \dots$. The ordering begins with $6$, and thereafter the numbers are arranged according to order: $6k-5 \;\gg_f\; 9k-6  \;\gg_f\; 9k-3 \;\gg_f\; 6k-2 \;\gg_f\; 3k-1 \;\gg_f\; 9k, \quad  k \in \mathbb{N}$. We will call the ordering $\gg_f$ the \emph{fast ordering} of $\mathbb{N}_{>3}$. 

\smallskip
For each $m \in \mathbb{N}_{>3}$, define
\[
\mathcal{M}_f(m) = \{ n \in \mathbb{N}_{>3} : n \gg_f m \} \cup \{ m \}.
\]
\end{definition}

\begin{theorem}\label{forcing:mixing:red:main:theorem}
	Let $h \in \mathcal{R}$, and let $\mathcal{P}_f(h)$ denote the collection of periods corresponding to exactfast cycles of $h$.

	If $m, n \in \mathbb{N}_{>3}$ satisfy $m \gg_f n$ and $m \in \mathcal{P}_f(h)$, then there exists a neighborhood $N$ of $h$ in $\mathcal{R}$ such that for every $g \in N$, one has $n \in \mathcal{P}_f(g)$.
	
	Consequently, there exists $m \in \mathbb{N}_{>3}$ for which $\mathcal{P}_f(h) = \mathcal{M}_f(m)$ ,and for all $g \in \mathcal{R}$ sufficiently close to $h$, $\mathcal{M}_f(m) \subseteq \mathcal{P}_f(g)$.
\end{theorem}

\begin{proof}
	Observe that any natural number strictly greater than $3$ is one of  the following forms: $6k-5, \; 9k-6, \; 9k-3,  \; 6k-2, \; 3k-1, \; 9k$.  By Definition \ref{red:ordering}, $6k-5 \; \gg_f\; 9k-6  \; \gg_f\; 9k-3 \; \gg_f\; 6k-2 \; \gg_f\; 3k-1 \; \gg_f\; 9k  \; \gg_f\; 6(k+1)-5= 6k+1, k \in \mathbb{N}$. The least \emph{rotation numbers} greater than $\frac{1}{3}$ for \emph{patterns} of those periods are respectively, $\frac{2k-1}{6k-5}, \; \frac{3k-1}{9k-6}, \; \frac{3k}{9k-3}, \; \frac{2k}{6k-2}, \; \frac{k}{3k-1}, \; \frac{3k+1}{9k}, \; \frac{2k+1}{6k+1}$ and these numbers are ordered as $\frac{1}{3} < \dots \frac{2k+1}{6k+1} < \frac{3k+1}{9k} <  \frac{k}{3k-1} = \frac{2k}{6k-2} = \frac{3k}{9k-3} < \frac{3k-1}{9k-6} < \frac{2k-1}{6k-5} < \dots$.

	Let $6k-5 \in \mathcal{P}_f(h)$. By Theorem \ref{result:1}, $h$ has a \emph{fast} cycle $P^1$ with \emph{rotation pair} $(2k-1,6k-5)$. Since, $ \frac{1}{3} < \frac{3k-1}{9k-6} < \frac{2k-1}{6k-5}$, $h$ has a \emph{fast} cycle $Q^1$ with \emph{rotation pair} $(3k-1, 9k-6)$. By Theorem \ref{loops:orbits:connection:1}, $h$ has a \emph{point loop} $\gamma^1$ \emph{corresponding} to $Q^1$. By Theorem \ref{loop:maintainance}, there exists a neighborhood $N^1$ of $h$ such that for each $g \in N^1$, $g$ has the \emph{point loop} $\gamma^1$. Now, by Theorem \ref{loops:orbits:connection:1}, each $g \in N^1$, has a cycle $Q_g^1$ corresponding to $\gamma$ with \emph{rotation pair}  $(3k-1, 9k-6)$. Since, $3k-1$ and $9k-6$ are co-prime, by Theorem \ref{block:structure:necessary}, $Q_g^1$ has \emph{no block structure}. This means $Q_g^1$ is a  \emph{exact}  cycle by Theorem \ref{connection:mixing:no:block:structure}. Hence, $9k-6 \in \mathcal{P}_f(g)$ for each $g \in N^1$.

	  Let $9k-6 \in \mathcal{P}_f(h)$. By Theorem \ref{result:1}, $h$  has a \emph{twist fast} cycle $P^2$ with \emph{rotation pair} $(3k-1,9k-6)$.  By Theorem \ref{k:n:red:forces:k+1:n+3}, there exists a \emph{neighborhood} $N^2$ of $h$, such that for each $g \in N^2$, $g$ has a \emph{exact fast} cycle $Q_g^2$ with \emph{rotation pair} $(3k,9k-3)$  and hence, $9k-3\in \mathcal{P}_f(g)$ for every $g \in N^2$.

	  If $9k-3 \in \mathcal{P}_f(h)$, Theorem \ref{9k+3:6k+2}, guarantees the existence of a neighborhood $N^3$ of $h$ such that for each $g \in N^3$, $6k-2 \in \mathcal{P}_f(g)$.
	  
	  	Let $6k-2 \in \mathcal{P}_f(h)$. By Theorem \ref{result:1}, $h$ has a \emph{fast} cycle $P^4$ with \emph{rotation pair} $(2k,6k-2)$. Its \emph{modified rotation pair}  is $(\frac{k}{3k-1}, 2)$. By Theorem \ref{result:1}, $h$ has a \emph{cycle} $Q^4$ with  \emph{modified rotation pair}   $(\frac{k}{3k-1}, 1)$.  Since, $k$ and $3k-1$ are co-prime, by Theorem \ref{block:structure:necessary}, $Q^4$ has \emph{no block structure}. Now, by Theorems 	\ref{loops:orbits:connection:1}, \ref{loop:maintainance} and \ref{connection:mixing:no:block:structure}, there exists a neighborhood $N^4$ of $h$ such that for each $g \in N^4$, $3k-1 \in \mathcal{P}_f(g)$.

	 Let $3k-1 \in \mathcal{P}_f(h)$. By Theorem \ref{result:1}, $h$ has a \emph{fast} cycle $P^5$ with \emph{rotation pair} $(k,3k-1)$. Since, $ \frac{1}{3} < \frac{3k+1}{9k} < \frac{k}{3k-1}$, by Theorem \ref{result:1},  $h$ has a \emph{fast} cycle $Q^5$ with \emph{rotation pair} $(3k+1, 9k)$. Since, $3k+1$ and $9k$ are co-prime, by Theorems \ref{block:structure:necessary}, 	\ref{loops:orbits:connection:1}, \ref{loop:maintainance} and \ref{connection:mixing:no:block:structure}, there exists a neighborhood $N^5$ of $h$ such that for each $g \in N^5$, $9k \in \mathcal{P}_f(g)$.

	 Let $9k \in \mathcal{P}_f(h)$. By Theorem \ref{result:1}, $h$ has a \emph{fast} cycle $P^6$ with \emph{rotation pair} $(3k+1,9k)$. Since, $ \frac{1}{3} < \frac{2k+1}{6k+1} < \frac{3k+1}{9k}$, by Theorem \ref{result:1},  $h$ has a \emph{fast} cycle $Q^6$ with \emph{rotation pair} $(2k+1, 6k+1)$. Since, $2k+1$ and $6k+1$ are coprime, by Theorems \ref{block:structure:necessary}, 	\ref{loops:orbits:connection:1}, \ref{loop:maintainance} and \ref{connection:mixing:no:block:structure}, there exists a neighborhood $N^6$ of $h$ such that for each $g \in N^6$, $6k+1 = 6(k+1)-5 \in \mathcal{P}_f(g)$.

	  Now, the result follows from the \emph{transitivity} of \emph{forcing}. 
\end{proof}

\subsection{Forcing among exact ternary patterns}\label{ternary:forcing:section}

Now, we will study the \emph{forcing} relations among \emph{exact ternary patterns} (See Figure \ref{ternary_mixing_six}).

\begin{theorem}\label{ternary:multiple}
	Let $P$ be a regular exact ternary cycle of a $P$-linear map $f \in \mathcal{R}$ of period $q>3$. Then $q$ is a multiple of $3$ and if $n_r$ and $n_g$ be the number of red and  green points of $P$, then,  $n_r=n_g \geqslant 1$. 
\end{theorem}

\begin{proof}
	Since $P$ is \emph{ternary}, clearly $q$ must be a multiple of $3$ and  simple computation yields $n_g=n_r$. Now, if $n_g = n_r =0$, that is, $P$ has only \emph{black} points,  then it is easy to see that $P$ has \emph{block structure} over the \emph{primitive} cycle of period $3$ and hence cannot be an \emph{exact}  map by Theorem \ref{connection:mixing:no:block:structure}.

\end{proof}

\begin{theorem}\label{ternary:add:3}
	Let $P$ be an exact ternary cycle of period $n$ of a $P$-linear map $f \in \mathcal{R}$. Then there exists a neighborhood $N$ of $f$  such that for each $h \in N$, $h$ has an exact ternary cycle $Q_h$ of period $n+3$. 
\end{theorem}

\begin{proof}
	As before, we will assume that the \emph{branches} of $\tau$ have been \emph{canonically ordered} such that the point $p_i$, $i=0,1,2$ of $P$ closest to the \emph{branching point} $a$ in each \emph{branch} $b_i$, $i=0,1,2$ is \emph{black}.  Let $\alpha$ be the \emph{fundamental point loop associated} with $P$. Then, $\alpha$ passes exactly once through $p_0$. Let $\beta$ be the \emph{point} loop of length $n+3$ obtained by adjoining $\alpha$ with the the \emph{black} loop of length three,  $\gamma: p_0 \to p_1 \to p_0$.  It is easy to see that $\beta$ has \emph{rotation number} $\frac{1}{3}$. By Theorem \ref{loop:maintainance}, there exists a neighborhood $N_1$ of $f$ such that for each $h \in N_1$, $h$ has the \emph{point loop} $\beta$. Let $Q_h$ be the cycle of $h$  associated with $\beta$. If $Q_h$ has no \emph{block structure}, by Theorem \ref{connection:mixing:no:block:structure},  $Q_h$ is an \emph{exact ternary} cycle of period $n+3$ and we are done.
	
	Suppose $Q_h$ has \emph{block structure}. Let $n_r$ be the number of \emph{red} points of $P$. By Theorem \ref{ternary:multiple}, $n_r \geqslant 1$. Choose a \emph{red} point $x_r$ of $P$ \emph{farthest} from $a$ on its \emph{branch}.  By Theorem \ref{black:loop:length:3} (1), there exists a \emph{black} loop $\delta$ of length $3$ passing through $x_r$. Let $\beta'$ be the \emph{point} loop of length $n+3$ obtained by adjoining $\alpha$ with $\delta$.  By Theorem \ref{loop:maintainance}, there exists a neighborhood $N_2$ of $f$ such that for each $g \in N_2$, $g$ has the \emph{point loop} $\beta'$. Let $R_g$ be the cycle of $g$  associated with $\beta'$. It is  easy to see that,  $R_g$ is \emph{ternary}, has period $n+3$ and  \emph{no block structure}. Hence by Theorem \ref{connection:mixing:no:block:structure}, $R_g$ is an \emph{exact ternary} cycle of period $n+3$, and hence the result follows. 
\end{proof}

	Let $\mathbb{N}_{3}$ denote the set of all positive integers that are multiples of $3$ and strictly greater than $3$. For each $m \in \mathbb{N}_{3}$, define \[\mathcal{M}_t(m) = \{ n \in \mathbb{N}_{3} : n > m \} \cup \{ m \} \].

	\begin{figure}[H]
		\caption{An \emph{exact ternary} cycle of period $6$}
		\centering
		\includegraphics[width=0.5 \textwidth]{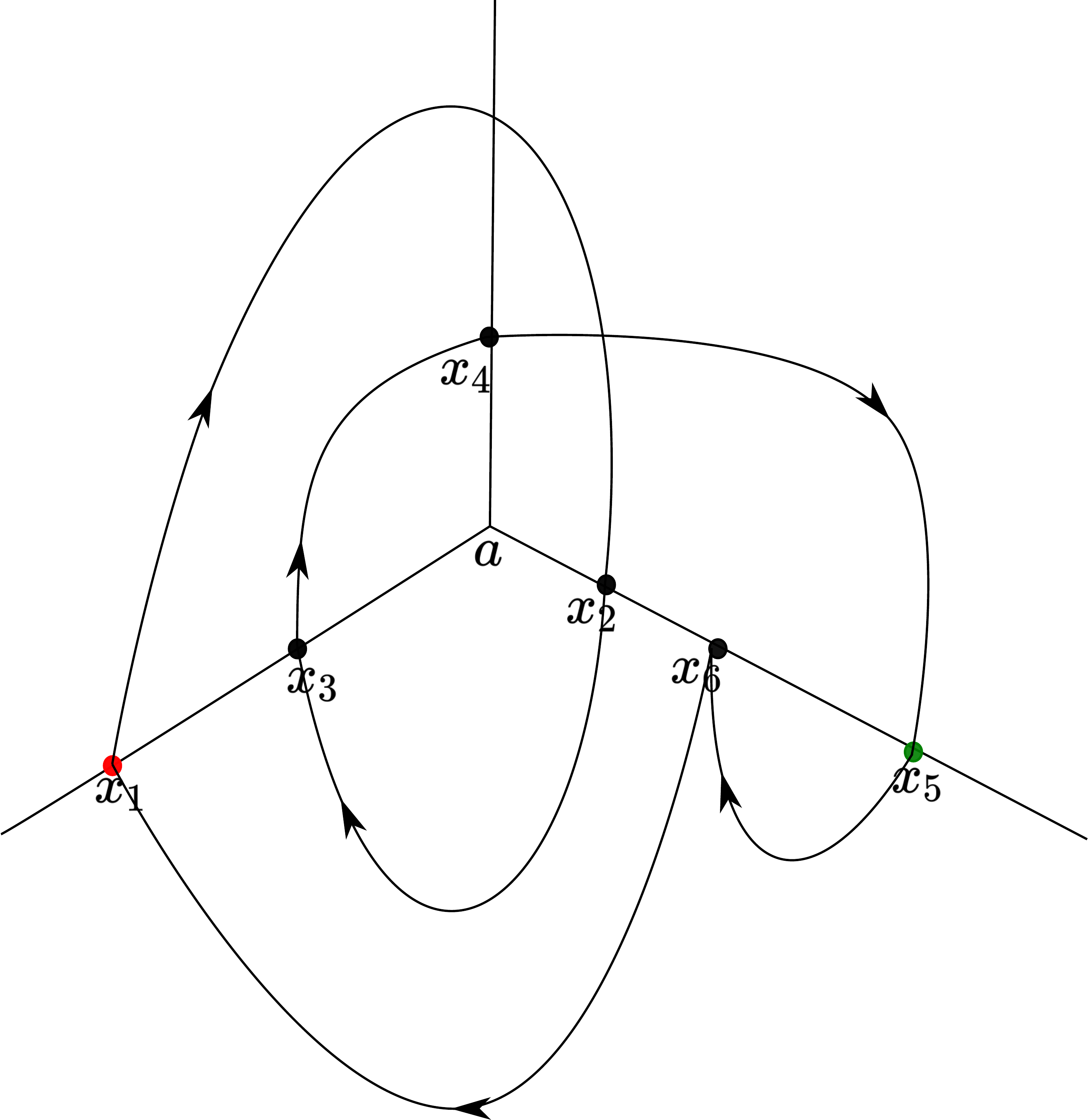}
		\label{ternary_mixing_six}
	\end{figure}

\begin{theorem}\label{ternary:final}

	Let $f \in \mathcal{R}$, and let $\mathcal{P}_t(f)$ denote the set of all periods corresponding to exact ternary cycles of $f$.
	
	If $m, n \in \mathbb{N}_{3}$ satisfy $m > n$ and $n \in \mathcal{P}_t(f)$, then there exists a neighborhood $N$ of $f$ in $\mathcal{R}$ such that $m \in \mathcal{P}_t(g)$ for every $g \in N$.
	
	Equivalently, there exists $m \in \mathbb{N}_{3}$ for which $\mathcal{P}_t(f) = \mathcal{M}_t(m)$,
	and for all $g$ sufficiently close to $f$, one has $\mathcal{M}_t(m) \subseteq \mathcal{P}_t(g)$.
\end{theorem}

\begin{proof}
	Follows from Theorems \ref{ternary:multiple} and \ref{ternary:add:3}. 
	
\end{proof}

A natural direction for future research concerns the behavior of \emph{non-exact} patterns on triods. By Theorem~\ref{block:structure:necessary}, such patterns necessarily possess a \emph{block structure}. An important question, therefore, is to determine the precise rule governing the coexistence of these \emph{non-exact patterns} and to characterize the corresponding \emph{forcing relations} among their \emph{periods}. We plan to investigate these questions in our forthcoming work.

\end{document}